\documentclass[12pt]{amsart}
\usepackage{tikz, hyperref, amssymb,amsmath,enumerate,mathtools}
\usetikzlibrary{decorations.markings}
\usepackage[mathscr]{eucal}

\hypersetup{colorlinks=true, linkcolor=blue, linkbordercolor=blue, citecolor=blue, citebordercolor=blue, linktocpage=true}

\mathtoolsset{centercolon}

\usepackage[all]{xy}
\usepackage{hyperref}
\usepackage{amssymb, color, amsmath, comment, mathrsfs}
\usepackage[T1]{fontenc}
\usepackage{tikz, epsfig}
\usepackage[normalem]{ulem}
\usepackage{cancel}
\usepackage{epsfig}
\definecolor{darkgreen}{rgb}{0.008,0.417,0.067}

\numberwithin{equation}{section}

\newtheorem{thm}{Theorem}[section]
\newtheorem{cor}[thm]{Corollary}
\newtheorem{lem}[thm]{Lemma}
\newtheorem{prp}[thm]{Proposition}

\theoremstyle{definition}

\theoremstyle{remark}

\numberwithin{equation}{section}

\newcommand{\CC}{\mathbb{C}}
\newcommand{\NN}{\mathbb{N}}

\newcommand{\ZZ}{\mathbb{Z}}

\newcommand{\Ff}{\mathcal{F}}

\newcommand{\Kk}{\mathcal{K}}
\newcommand{\Ll}{\mathcal{L}}
\newcommand{\Mm}{\mathcal{M}}
\newcommand{\Oo}{\mathcal{O}}

\newcommand{\Tt}{\mathcal{T}}

\newcommand{\eps}{\varepsilon}

\newcommand{\nn}{\mathbf{n}}

\newcommand{\Aut}{\operatorname{Aut}}
\newcommand{\clsp}{\overline{\lsp}}
\newcommand{\coker}{\operatorname{coker}}

\newcommand{\id}{\operatorname{id}}

\newcommand{\lsp}{\operatorname{span}}

\newcommand{\Img}{\operatorname{Img}}
\newcommand{\sign}{\operatorname{sign}}

\newcommand{\rg}{\operatorname{rg}}

\newcommand{\varphiY}{\varphi^Y}

\newcommand{\hatimes}{\mathbin{\widehat{\otimes}}}

\newcommand{\KHgr}[1]{K_{\operatorname{gr}}^{#1}}
\newcommand{\KTgr}[1]{K^{\operatorname{gr}}_{#1}}

\title[Graded $K$-theory and $K$-homology of Cuntz--Pimsner algebras]{Graded $K$-theory and $K$-homology of relative Cuntz--Pimsner algebras and graph $C^*$-algebras}

\author{Quinn Patterson}
\address[Q.~Patterson, A.~Sierakowski, A.~Sims and J.~Taylor]{School of Mathematics and Applied Statistics\\
	University of Wollongong\\
	NSW  2522\\
	AUSTRALIA}
\email{qrp844@uowmail.edu.au, asierako@uow.edu.au, asims@uow.edu.au, \newline jpt812@uowmail.edu.au}
\author{Adam Sierakowski}
\author{Aidan Sims}
\author{Jonathan Taylor}

\keywords{$KK$-theory; graded $K$-theory; $C^*$-algebra; graded $C^*$-algebra}
\subjclass[2010]{Primary 46L05, 19K35}
\thanks{We thank Ralf Meyer for his helpful comments and suggestions. This research was supported by ARC Discovery Project DP180100595.}
\date{\today}

\begin{document}

\begin{abstract}
We establish exact sequences in $KK$-theory for graded relative Cuntz--Pimsner algebras
associated to nondegenerate right-Hilbert bimodules. We use this to calculate the graded
$K$-theory and $K$-homology of relative Cuntz--Krieger algebras of directed graphs for gradings
induced by $\{0,1\}$--valued labellings of their edge sets.
\end{abstract}

\maketitle

\section*{Introduction}

In the study of $C^*$-algebras, operator $K$-theory, which generalises topological $K$-theory via
Gelfand duality, has long been a key invariant. Indeed, it is a knee-jerk reaction for
$C^*$-algebraist these days, when presented with a new example, to try to compute its $K$-theory;
and the $K$-theory frequently reflects key structural properties. For example, the ordered
$K$-theory of an irrational rotation algebra recovers the angle of rotation up to a minus sign
\cite{Elliott, E}, while the $K$-theory of a Cuntz--Krieger algebra recovers the Bowen--Franks
group of the associated shift space \cite{CunKri, Cun81}. Cuntz proved that the $K$-theory groups
of the $C^*$-algebra of a finite directed graph $E$ with no sources and with $\{0,1\}$-valued
adjacency matrix $A_E$ are the cokernel and kernel of the matrix $1 - A^t_E$ regarded as an
endomorphism of the free abelian group $\ZZ E^0$ (see \cite[Proposition~3.1]{Cun81}). This was
generalised to row-finite directed graphs $E$ with no sources in \cite{PR}, to all row-finite
directed graphs $E$ in \cite{RSz}, and to arbitrary graphs (with appropriate adjustments made to
the domain and the codomain of $A^t_E$) in \cite{BHRS, DT}, see also \cite{EL2, Yi, Szym, Szym02,
Raeburn}.

Dual to $K$-theory is the $K$-homology theory that emerged in the pioneering work of
Brown--Douglas--Fillmore \cite{BDF1, BDF2}. It is less of an automatic reaction to compute
$K$-homology for $C^*$-algebras, but, for example, Cuntz and Krieger computed (in
\cite[Theorem~5.3]{CunKri}) the $\operatorname{Ext}$-group (that is, the odd $K$-homology group)
of the Cuntz--Krieger algebra $\Oo_A$ of $A \in M_n(\ZZ_+)$ as the cokernel of $1 - A$ regarded as
an endomorphism of $\ZZ^n$. The computation was later generalised to graph $C^*$-algebras in
\cite{Tom, DT} (see also \cite{Yi, Crisp}).

Both $K$-theory and $K$-homology are unified in Kasparov's $KK$-theory \cite{K1, K2}: the
$K$-theory and $K$-homology groups of $A$ are the Kasparov groups $KK_*(\CC, A)$ and $KK_*(A,
\CC)$ respectively. So Kasparov's theory provides a unified approach to calculating $K$-theory and
$K$-homology. Pimsner exploited this in \cite{P}, developing two exact sequences in $KK$-theory
for the $C^*$-algebra $\Oo_X$ associated to a right-Hilbert $A$--$A$-bimodule $X$, and using one
of them to compute the $K$-theory of $\Oo_X$ in terms of that of $A$. Every graph determines an
associated Hilbert module, and while the Pimsner algebra of this module only agrees with the graph
$C^*$-algebra when the graph has no sources, Muhly and Tomforde developed a modified bimodule
\cite{MT} whose Pimsner algebra always contains the graph $C^*$-algebra as a full corner. In
particular, combining these results provides a new means of computing the $K$-theory and
$K$-homology of graph $C^*$-algebras.

Kasparov's $KK$-theory is most naturally a theory for graded $C^*$-algebras, and the results
described above are obtained by endowing the $C^*$-algebras involved with the trivial grading.
However, graph $C^*$-algebras admit many natural gradings: by the universal property of $C^*(E)$,
every binary labelling $\delta : E^1 \to \{0,1\}$ of the edges of $E$ induces a grading
automorphism that sends the generator $s_e$ associated to an edge $e$ to $(-1)^{\delta(e)} s_e$.
More generally, every grading of a right-Hilbert bimodule induces a grading of the associated
Pimsner algebra.

In \cite{KPS6}, Kumjian, Pask and Sims investigated graded $K$-theory and $K$-homology, defined in
terms of Kasparov theory (other approaches to graded $K$-theory are investigated in, for example,
\cite{vD1, vD2, Karoubi}) of graded graph $C^*$-algebras, extending earlier results of \cite{H1,
H2} for the Cuntz algebras $\Oo_n$. By extending Pimsner's arguments to $C^*$-algebras of graded
Hilbert bimodules with injective left actions by compacts, they computed the graded $K$-theory of
the $C^*$-algebras of row-finite graphs with no sources. They showed (in
\cite[Collollary~8.3]{KPS6}) that if $E$ is a row-finite directed graph with no sources,
$\alpha_\delta$ is the grading associated with a given function $\delta : E^1 \to \{0,1\}$, and
$A_E^\delta$ is the $E^0 \times E^0$ matrix with entries $A^\delta_E(v,w) = \sum_{e \in v E^1 w}
\delta(e)$, then the graded $K$-theory groups are isomorphic to the cokernel and kernel of $1 -
(A^\delta_E)^t$ regarded as an endomorphism of $\ZZ E^0$.

In this paper we compute both the graded $K$-theory and the graded $K$-homology of relative
Cuntz--Krieger algebras of arbitrary graphs: Let $V$ be any subset of \emph{regular vertices} $E^0_{\rg}$ (those
which receive a nonzero finite set of edges). The relative Cuntz--Krieger algebra $C^*(E; V)$ is
the universal $C^*$-algebra in which the Cuntz--Krieger relation is only imposed at vertices in
$V$. In particular  $C^*(E)=C^*(E; E^0_{\rg})$ for directed graph $E$. Let $A^\delta_V$ be the $V \times
E^0$ matrix with entries given by the same formula as $A^\delta_E$, regarded as a homomorphism
from $\ZZ E^0$ to $\ZZ V$. Write $\widetilde{A}^\delta_V$ for the dual homomorphism from
$\ZZ^{E^0}$ to $\ZZ^V$. Let $\iota : \ZZ V \to \ZZ E^0$ be the inclusion map, and let $\pi :
\ZZ^{E^0} \to \ZZ^V$ be the projection map. Then the graded $K$-theory groups and $K$-homology
groups of the relative Cuntz--Krieger algebra are given by
\begin{align*}
\KTgr0(C^*(E; V),\alpha_\delta) &\cong \coker(\iota - A^\delta_V)^t, \quad  &\KTgr1(C^*(E; V),\alpha_\delta) &\cong \ker(\iota - A^\delta_V)^t,\\
\KHgr0(C^*(E; V),\alpha_\delta) &\cong \ker(\pi - \widetilde{A}^\delta_V), \quad  &\KHgr1(C^*(E; V),\alpha_\delta) &\cong \coker(\pi - \widetilde{A}^\delta_V).
\end{align*}

To prove this, we use that $C^*(E; V)$ may be realised as a relative Cuntz--Pimsner algebra of a
graph module $X(E)$. We verify that the two assumptions (namely injectivity and compactness)
imposed on the left actions of $A$ on a Hilbert module
$X$ in the arguments of \cite{P, KPS6} are not needed. As a result we obtain exact sequences in
$KK$-theory analogous to those of \cite{KPS6} for relative Cuntz--Pimsner algebras. By calculating
the $KK$-groups and the maps between them in the situation where $X$ is the graph module $X(E)$,
we obtain the desired calculations of graded $K$-theory and $K$-homology for relative graph
$C^*$-algebras, substantially generalising the results in \cite{KPS6}.

We then present an alternative calculation using Muhly and Tomforde's adding-tails construction.
In this approach, Pimsner's exact sequences are needed only for modules where the homomorphism
implementing the left action is injective. Given an arbitrary nondegenerate bimodule $X$, we add an infinite
direct sum of copies of the Katsura ideal $J_X$ to both the coefficient algebra and the module $X$ to
obtain a new module $Y$ over a new algebra $B$ which acts injectively on the left. We then recover
the exact sequences for $\Oo_X$ from the ones we already have for $\Oo_Y$ using countable
additivity in $KK$-theory. This is automatic in the first variable, so we obtain a complete
generalisation of the contravariant exact sequence of \cite{KPS6} for $\Oo_X$. However,
$KK$-theory is not in general countably additive in the second variable. However
$KK(\CC, \cdot)$ is countably additive for graded $C^*$-algebras (we could not find a reference , so we give the details) so we obtain an exact sequence
describing the graded $K$-theory $KK_i(\CC, \Oo_X)$.

\smallskip

We begin in Section~\ref{sec:background} with some background on $KK$-theory, mostly to establish
notation. More-detailed background on $KK$-theory can be found in \cite{KPS6}, and of course in
Blackadar's book \cite{B}, which is our primary reference. We assume the reader is familiar with
Hilbert modules and graph $C^*$-algebras. A convenient summary of the requisite background appears
in \cite{KPS6}, and more details can be found in \cite{Raeburn, tfb, Lance, P}. We also provide a
little background on the relative Cuntz--Pimsner algebras of Muhly and Solel \cite{MS}, of which
Katsura's Katsura--Pimsner algebras are a special case. In Section~\ref{sec:noncompact} we show
how to generalise the results of \cite{KPS6} to relative Cuntz--Pimsner algebras of arbitrary
essential graded Hilbert modules. In Section~\ref{sec:graph algebras} we apply these results to
compute the graded $K$-theory and $K$-homology of relative Cuntz--Krieger algebras of arbitrary
graphs. In Section~\ref{sec:adding tails}, we show how Muhly and Tomforde's adding-tails
construction for Hilbert modules can be adapted to graded modules, and reconcile this with our
$K$-theory and $K$-homology results for the graded Katsura--Pimsner algebra of a nondegenerate
Hilbert module.

\section{Background material}\label{sec:background}

In this section we provide some background on $KK$-theory, relative graph $C^*$-algebras and terminology
used in the later sections.

\subsection{Direct sums and products of groups} Let $S$ be any countable set. We let $\ZZ S$ denote
the direct sum $\bigoplus_{s\in S} \ZZ$ of copies of $\ZZ$ (the group of all finitely supported
functions from $S$ to $\ZZ$), and $\ZZ^S$ denotes the direct product $\prod_{s\in S}\ZZ$ of copies
of $\ZZ$ (the group of all functions from $S$ to $\ZZ$). More generally we write
$\prod^\infty_{n=1}G_n$ for the infinite product of groups $G_n$, and $\bigoplus^\infty_{n=1} G_n$
for the subgroup generated by the $G_n$.

\subsection{Hilbert modules} Given a $C^*$-algebra $B$ and a right Hilbert $B$-module $X$, we
write $\Ll(X)$ for the adjointable operators on $X$, we write $\Kk(X)$ for the generalised compact
operators on $X$, and given $\xi,\eta \in X$ we write $\Theta_{\xi,\eta}$ for the compact operator
$\zeta \mapsto \xi \cdot \langle \eta, \zeta\rangle_B$. If $\phi : A \to \Ll(X)$ is a
$C^*$-homomorphism so that $X$ is an $A$--$B$-correspondence, we say that the left action is
\emph{injective} if $\phi$ is injective and that the left action is \emph{by compact operators} if
$\phi(A) \subseteq \Kk(X)$.

Let $I$ be an ideal of a $C^*$-algebra $A$ and $X$ be a right Hilbert $A$-module $X$. Following \cite{Kat},
we define $XI:=\{x\in X:\langle x,x\rangle\in I\}$. This $XI$ is a right Hilbert $I$-module under the same operations as
$Y$, and $XI = X\cdot I:=\{x \cdot i : x \in X, i \in I\}$ justifying the notation.

\subsection{Gradings} A \emph{grading} of a $C^*$-algebra is a self-inverse automorphism $\alpha$ of $A$, and
decomposes $A$ into direct summands $A_0 = \{a : \alpha(a) = a\}$ and $A_1 = \{a : \alpha(a) =
-a\}$. We write $\partial(a) = i$ if $a \in A_i$. If $a,b$ are homogeneous, then their graded
commutator is $[a,b]^{\operatorname{gr}} := ab - (-1)^{\partial(a)\partial(b)}ba$, and we extend
this formula bilinearly to arbitrary $a,b \in A$. Elements of $A_0 \cup A_1$ are called
\emph{homogeneous}, elements of $A_0$ are \emph{even} and elements of $A_1$ are \emph{odd}. A
homomorphism of graded $C^*$-algebras is a graded homomorphism if it intertwines the grading
automorphisms.

A \emph{grading} of a $C^*$-correspondence $X$ over graded $C^*$-algebras $(A, \alpha_A)$ and $(B,\alpha_B)$
is a map $\alpha_X : X \to X$ such that $\alpha_X^2 = \id$, $\alpha_X(a \cdot x \cdot b) =
\alpha_A(a)\cdot\alpha_X(a)\cdot\alpha_B(b)$, and $\alpha_B(\langle x,
y\rangle_B) = \langle\alpha_X(x), \alpha_X(y)\rangle_B$. This induces a grading
$\tilde{\alpha}_X$ on $\Ll(X)$ given by $\tilde{\alpha}_X(T) = \alpha_X \circ T \circ \alpha_X$.
Again, $X$ decomposes as the direct sum of $X_0 = \{\xi : \alpha_X(\xi) = \xi\}$ and $X_1 = \{\xi
: \alpha_X(\xi) = -\xi\}$, we call the elements of the $X_i$ homogeneous, and odd and even as
above.

For a graded $C^*$-algebra $B$ set $\mathbb{I}B:=C([0,1]) \hatimes B$ with the trivial grading on
$C([0,1])$. For each $t\in [0,1]$ the homomorphism $\epsilon_t : \mathbb{I}B \to B$ given by
$\epsilon_t(f\hatimes b) = f(t)b$ is graded. It follows that $B$ may be regarded as a graded
$\mathbb{I}B$--$B$-correspondence $(\epsilon_t, B_B)$.

\subsection{Graded tensor products} The \emph{graded tensor product} of graded $C^*$-algebras $A$ and $B$ is the minimal
$C^*$-completion of the algebraic tensor product $A \odot B$ with involution satisfying $(a
\hatimes b)^* = (-1)^{\partial a \cdot\partial b} a^* \hatimes b^*$ and multiplication satisfying
$(a \hatimes b)(a' \hatimes b') = (-1)^{\partial b \cdot \partial a'} aa' \hatimes bb'$ for
homogeneous elements $a,a' \in A$ and $b, b' \in B$. There is a grading $\alpha_A \hatimes \alpha_B$ on
$A \hatimes B$ such that $(\alpha_A \hatimes\alpha_B)(a \hatimes b) = \alpha_A(a)
\hatimes\alpha_B(b)$. The balanced tensor product $X \otimes_B Y$ of graded $C^*$-correspondences
admits a grading $\alpha_X\hatimes \alpha_Y$ on $X \otimes_\psi Y$ such that $(\alpha_X
\hatimes\alpha_Y)(x \hatimes y) = \alpha_X(x) \hatimes\alpha_Y(y)$.

\subsection{Relative Cuntz--Pimsner algebras} Let $X$ be a graded $A$--$A$-correspondence,
and write $\varphi : A \to \Ll(X)$ for the homomorphism inducing the left action.

A \emph{representation} of $X$ in a $C^*$-algebra $B$ is a pair $(\pi, \psi)$ consisting of a
homomorphism $\pi : A \to B$ and a linear map $\psi : X \to B$ such that $\pi(a)\psi(\xi)\pi(b) =
\psi(a \cdot x \cdot b)$ and $\pi(\langle \xi, \eta\rangle_A) = \psi(\xi)^*\psi(\eta)$ for all
$a,b \in A$ and $\xi,\eta \in X$. The \emph{Toeplitz algebra} $\Tt_X$ of $X$ is the universal
$C^*$-algebra generated by a representation of $X$. Such a representation induces a homomorphism
$\psi^{(1)} : \Kk(X) \to B$ satisfying $\psi^{(1)}(\Theta_{\xi,\eta}) = \psi(\xi)\psi(\eta)^*$ for
all $\xi,\eta$. The universal property of $\Tt_X$ gives a grading $\alpha_\Tt$ of $\Tt_X$ such
that if $(\pi, \psi)$ is the universal representation of $X$ in $\Tt_X$, then $\alpha_\Tt \circ
\pi = \pi \circ \alpha_A$, and $\alpha_\Tt \circ \psi = \psi \circ \alpha_X$.

Let $C := \varphi^{-1}(\Kk(X))$; observe that if $(\pi,\psi)$ is a representation of $X$, then
both $\pi$ and $\psi^{(1)} \circ \varphi$ are homomorphisms from $C$ to $\Tt_X$. Given an ideal $I
\subseteq C$, the relative Cuntz--Pimsner algebra $\Oo_{X, I}$ is defined to be the universal
$C^*$-algebra generated by a representation $(\pi,\psi)$ that is \emph{$I$-covariant} in the sense
that $\pi|_I = (\psi^{(1)} \circ \varphi)|_I$. If $I$ is invariant under the grading $\alpha_A$ of
$A$, then the universal property of $\Oo_{X, I}$ shows that the grading $\alpha_\Tt$ of $\Tt_X$
descends to a grading $\alpha_\Oo$ of $\Oo_{X, I}$.

The \emph{Fock space} $\Ff_X$ is the internal direct sum $\Ff_X := \bigoplus^\infty_{n=0}
X^{\hatimes n}$, with the convention that $X^{\hatimes 0} = {_A A_A}$. There is a representation
$(\ell_0, \ell_1)$ of $X$ in $\Ll(\Ff_X)$ such that $\ell_0(a)\xi = a \cdot \xi$ and such that
$\ell_1(\xi)\eta = \xi \otimes_A \eta$. The induced homomorphism $\pi_0 : \Tt_X \to
\Ll(\Ff_X)$ is injective.

The ideal $C = \varphi^{-1}(\Kk(X))$ of $A$ induces the submodule $\Ff_{X, C} := \Ff_X C$.
The subalgebra $\Kk(\Ff_{X, C}) := \clsp\{\Theta_{\xi,\eta} :
\xi,\eta \in \Ff_{X, C}\} \subseteq \Ll(\Ff_X)$ is contained in $\pi_0(\Tt_X)$ (see
\cite[Lemma~2.17]{MS}). Since $\pi_0 : \Tt_X \to \Ll(\Ff_X)$ is injective, it determines an
inclusion $j : \Kk(\Ff_{X, C}) \to \Tt_X$. In particular, for any ideal $I \subseteq C$, $j$
restricts to a graded inclusion of $\Kk(\Ff_{X, I})$ in $\Tt_X$. Theorem~2.19 in \cite{MS}
and a routine application of universal properties show that the quotient map $\Tt_{X}\to \Oo_{X, I}$ induces an isomorphism $\Tt_{X}/j(\Kk(\Ff_{X, I})) \cong \Oo_{X, I}$.

\subsection{Kasparov modules}\label{sec:KKBasics} If $(A, \alpha_A)$ and $(B, \alpha_B)$ are separable graded $C^*$-algebras, then a \emph{Kasparov
$A$--$B$-module} is a quadruple $(X, \phi, F, \alpha_X)$ where $(\phi, X)$ is a countably
generated $A$--$B$-correspondence, $\alpha_X$ is a grading of $X$, and $F \in \Ll(X)$
is an odd element with respect to the grading $\tilde\alpha_X$ on $\Ll(X)$ such that for all
$a\in A$ the elements $(F  - F^*)\phi(a)$, $(F^2 - 1)\phi(a)$, and $[F,
\phi(a)]^{\operatorname{gr}}$ are compact. When these elements are all zero we
call $(X, \phi, F, \alpha_X)$ for a \emph{degenerate} Kasparov module.

Kasparov $A$--$B$-modules $(X_0, \phi_0, F_0, \alpha_{X_0})$ and $(X_1, \phi_1, F_1,
\alpha_{X_1})$ are \emph{unitarily equivalent} if there is a unitary $U \in \Ll(X_0, X_1)$ that
intertwines $\phi_0$ and $\phi_1$, $F_0$ and $F_1$, and $\alpha_0$ and $\alpha_1$. They are
\emph{homotopy equivalent} if there is Kasparov $A$--$\mathbb{I}B$-module $(X, \phi, F, \alpha_X)$
such that, $(X \hatimes_{\epsilon_i} B_B, \tilde{\epsilon_i}\circ\phi, \tilde{\epsilon_i}(F),
\alpha_X \hatimes\alpha_B)$ is unitarily equivalent to $(X_i, \phi_i, F_i, \alpha_{X_i})$ for each
of $i = 0,1$. Homotopy equivalence is denoted $\sim_h$, and is an equivalence relation. The
Kasparov group $KK(A,B)$ is the collection of all homotopy classes of Kasparov $A$--$B$-modules,
which is a group under the operation induced by taking direct sums of Kasparov modules. Given a
graded homomorphism $\psi : A \to B$ of $C^*$-algebras, and a Kasparov $B$--$C$-module $(X, \phi,
F, \alpha_X)$, we obtain a new Kasparov $A$--$C$-module $(X, \phi \circ \psi, F, \alpha_X)$, whose
class we denote $\psi^*[X]$. For a graded homomorphism $\psi\colon B\to C\cong \Kk(C_C)$ we let
$[\psi]:=[C_C, \psi, 0, \alpha_C]\in KK(B,C)$. If $\phi\colon A\to B$ is a graded homomorphism and
$(Y, \psi, G, \alpha_Y)$ is a Kasparov $C$--$A$-module, then $(Y \hatimes_\phi B_B, \psi
\hatimes1, G \hatimes1, \alpha_Y \hatimes\alpha_B)$ is a Kasparov $C$--$B$-module whose class we
denote by $\phi_*[Y, \psi, G, \alpha_Y]$.

We will need to use that graded Morita equivalence implies $KK$-equivalence. Suppose that $(A,
\alpha_A)$ and $(B, \alpha_B)$ are graded $C^*$-algebras, and $(X, \alpha_X)$ is a graded
imprimitivity $A$--$B$-module. Then in particular the left action of $A$ on $X$ is implemented by
a homomorphism $\varphi : A \to \Kk(X)$, and the right action of $B$ on the dual module $X^*$ is
implemented by a homomorphism $\psi : B \to \Kk(X^*)$. So we obtain $KK$-classes $[X] := [X,
\varphi, 0, \alpha_X] \in KK(A, B)$ and $[X^*] := [X^*, \psi, 0, \alpha_{X^*}] \in KK(B, A)$.
Since the Fredholm operators in both modules of these Kasparov modules are zero, we have $[X]
\hatimes_B [X^*] = [X \otimes_B X^*, \phi \otimes 1, 0, \alpha_X \otimes \alpha_{X^*}]$. Since $X$
is an imprimitivity module, we have $X \otimes X^* \cong A$ via $\xi \otimes \eta \mapsto
{_A\langle \xi, \eta\rangle}$, and it is routine to check that this isomorphism intertwines the
canonical left action $\tilde\varphi$ of $A$ on $X \otimes X^*$ with the left action $L_A$ of $A$
on itself by left multiplication, and intertwines $\alpha_X \otimes \alpha_{X^*}$ with $\alpha_A$.
So $[X] \hatimes_B [X^*] = [A, L_A, 0, \alpha_A] = [\id_A]$. Similarly $[X^*] \hatimes_A [X] =
[\id_B]$, and so $(A, \alpha_A)$ and $(B, \alpha_B)$ are $KK$-equivalent.

\section{Graded K-theory of relative Cuntz--Pimsner algebras}\label{sec:noncompact}

In this section we generalise the main results in Section 4 of \cite{KPS6} establishing exact
sequences in $KK$-theory for graded relative Cuntz--Pimsner algebras associated to an essential graded
$A$--$A$-correspondence. We do not assume the
action is injective, nor compact nor that $X$ is full, but we do assume that $X$ is \emph{essential} (or  \emph{nondegenerate}) in the sense
that $\overline{\varphi(A)X}=X$.

\smallskip
\textbf{Set-up.} \emph{Throughout this section we fix a graded, separable, nuclear $C^*$-algebra
$A$ and a graded countably generated essential $A$--$A$-correspondence $X$ with a left
action $\varphi$ and we fix an ideal $I \subseteq \varphi^{-1}(\Kk(X))$.}

\smallskip
For readers interested in the Katsura--Pimsner algebra $\Oo_{X}$ (\cite[Definition 3.5]{Kat}) we
recall that it coincides with the relative Cuntz--Pimsner algebra $\Oo_{X, I}$ for  $I = \varphi^{-1}(\Kk(X)) \cap \ker\varphi^\perp$, where $\ker\varphi^\perp=\{b \in A : b \ker\varphi = \{0\}\}$.

\smallskip

We present terminology of \cite{KPS6} relevant to Lemma~\ref{Lemma: A and TtX KKequiv}.
Let $\Ff_X$ be the Fock space of $X$, let $\alpha_X^\infty$ be the
diagonal grading on $\Ff_X$, and let $\varphi^\infty:A\to\Ll(\Ff_X)$ be the diagonal left action of $A$
on $\Ff_X$. Recall that $\Tt_X$ is the Toeplitz algebra associated to $X$, generated by $i_A(a)=
\varphi^\infty(a)$ and $i_X(\xi)= T_\xi$, and $\alpha_\Tt$ is the restriction of
$\tilde\alpha_X^\infty$ to $\Tt_X$. Let $\pi_i:\Tt_X\to\Ll(\Ff_X)$ be the representations  determined by
\[
\pi_0(T_\xi)\rho=\begin{cases}
\xi\hatimes\rho,&\rho\in X^{\hatimes n}, n\geq 1,\\
\xi\cdot \rho,&\rho\in A,
\end{cases}, \quad \pi_1(T_\xi)\rho=\begin{cases}
\xi\hatimes\rho,&\rho\in X^{\hatimes n}, n\geq 1,\\
0,&\rho\in A
\end{cases}
\]
As presented in \cite[Section~4]{KPS6} there is a Kasparov $\Tt_X$--$A$-module given by
\[
M=\left(\Ff_X\oplus\Ff_X,\pi_0\oplus(\pi_1\circ\alpha_\Tt),\left(\begin{smallmatrix}
0&1\\1&0
\end{smallmatrix}\right),\left(\begin{smallmatrix}
\alpha_X^\infty&0\\0&-\alpha_X^\infty
\end{smallmatrix}\right)\right).
\]
Recall, for the canonical inclusion $i_A:A\hookrightarrow \Tt_X$ and
a graded $C^*$-algebra $B$ we have
\begin{align*}
[i_A]&=[(\Tt_X, i_A, 0, \alpha_\Tt)]\in KK(A,\Tt_X), \ \text{and}&
[\id_{B}]&=[B, \id_{B}, 0, \alpha_B]\in KK(B,B).
\end{align*}
By \cite[Theorem~4.2]{KPS6}, if $\varphi$ is injective and by compact operators, then the Kasparov classes $[i_A]$ and $[M]$
are mutually inverse in the sense that $[i_A]\hatimes_{\Tt_X}[M]=[\id_A]$ and
$[M]\hatimes_A[i_A]=[\id_{\Tt_X}]$. We prove the
result is true without assuming $\varphi$ is injective or by compact operators.

\begin{lem}[cf. {\cite[Theorem~4.2]{KPS6}}]\label{Lemma: A and TtX KKequiv}
With notation as above, the pair $[i_A]$ and $[M]$ are mutually inverse. In particular, $(A,\alpha_A)$
and $(\Tt_X, \alpha_\Tt)$ are $KK$-equivalent.
\end{lem}
\begin{proof}
The argument in the proof of \cite[Theorem~4.2]{KPS6} showing that
$[i_A]\hatimes_{\Tt_X}[M]=[\id_A]$ does not require injectivity nor compactness of the left action
$\varphi$.

To show $[M]\hatimes_A[i_A]=[\id_{\Tt_X}]$, we adjust the proof of \cite[Theorem~4.2]{KPS6}. Let
$\pi_0':=\pi_0\hatimes 1_\Tt$ and $\pi_1':=(\pi_1\circ\alpha_\Tt)\hatimes1_\Tt$. By
\cite[Proposition~18.7.2(a)]{B},  identifying $(\Ff_X\oplus\Ff_X)\hatimes_A\Tt_X$ with $(\Ff_X\hatimes_A\Tt_X)\oplus(\Ff_X\hatimes_A\Tt_X)$,
\begin{align*}
[M]\hatimes_A[i_A]&=(i_A)_*[M]=\left[(\Ff_X\oplus\Ff_X)\hatimes_A\Tt_X,\pi_0' \oplus \pi_1', \left(\begin{smallmatrix}
0&1\\1&0
\end{smallmatrix}\right),\left(\begin{smallmatrix}\alpha_X^\infty\hatimes\alpha_\Tt&0\\0&-\alpha_X^\infty\hatimes\alpha_\Tt\end{smallmatrix}\right)\right].
\end{align*}
Since $X$ is essential we have $A\hatimes_A\Tt_X\cong\Tt_X$ as graded
$A$--$\Tt_X$-correspondences, and so $\alpha_\Tt$ defines a left action of $\Tt_X$ on
$A\hatimes\Tt_X$. Extending by zero, we get an action $\tau$ of $\Tt_X$ on $\Ff_X\hatimes\Tt_X$. The proof of \cite[Theorem~4.2]{KPS6} shows that
\[
\left[(\Ff_X\oplus\Ff_X)\hatimes_A\Tt_X,0\oplus\tau,\left(\begin{smallmatrix}
0&1\\1&0
\end{smallmatrix}\right),\left(\begin{smallmatrix}\alpha_X^\infty\hatimes\alpha_\Tt&0\\0&-\alpha_X^\infty\hatimes\alpha_\Tt\end{smallmatrix}\right)\right]=-[\id_{\Tt_X}],
\]
and hence
\begin{equation}\label{M tensor iA is identity}
(i_A)_*[M]-[\id_{\Tt_X}]=\left[(\Ff_X\oplus\Ff_X)\hatimes_A\Tt_X,\pi_0'\oplus(\pi_1'+\tau),\left(\begin{smallmatrix}
0&1\\1&0
\end{smallmatrix}\right),\left(\begin{smallmatrix}\alpha_X^\infty\hatimes\alpha_\Tt&0\\0&-\alpha_X^\infty\hatimes\alpha_\Tt\end{smallmatrix}\right)\right].
\end{equation}
We claim \eqref{M tensor iA is identity} is the class of a degenerate Kasparov module. To show this, for each $t\in[0,1]$
define $\psi_t:X\to\Ll(\Ff_X\hatimes\Tt_X)$ by
\[
\psi_t(\xi)=\cos(t\pi/2)(\pi_0'(\alpha_{\Tt}(T_\xi))-\pi_1'(T_\xi))+\sin(t\pi/2)\tau(\xi)+\pi_1'(T_\xi).
\]
With $\tilde\varphi^\infty:=\varphi^\infty\hatimes 1_{\Tt_X}\colon A\to \Ll(\Ff_X\hatimes_A\Tt_X)$ we have a
Toeplitz representation $(\tilde\varphi^\infty\circ\alpha_A,\psi_t)$ of $X$. Hence for each
$t\in[0,1]$ there is a homomorphism $\pi_t':\Tt_X\to\Ll(\Ff_X\hatimes_A\Tt_X)$ such that
$\pi_t'(T_\xi)=\psi_t(\xi)$ and $\pi_t'(a)=\tilde\varphi^\infty\circ\alpha_A(a)$ for all $\xi\in X$
and $a\in A$. We claim that $K_{t,\xi}:=(\pi_t'-\pi'_1)(T_\xi)$ is compact for each $\xi\in X$.
To see this, note that $K_{t,\xi}$ vanishes on $(\Ff_X\ominus
A)\hatimes_A\Tt_X=(\Ff_X\hatimes_A\Tt_X)\ominus(A\hatimes_A\Tt_X)$, and has range contained in the
subspace $A\hatimes_A\Tt_X$, thus we need only need to show that $K_{t,\xi}$ is compact on $A\hatimes_A\Tt_X$.
To show this recall that for an $A$--$B$-correspondence $Y$ with an
left action $\psi:A\to\Ll(Y)$, putting $J:=\psi^{-1}(\Kk(Y))$, for each $k\in\Kk(A J)$ the operator $k\hatimes1_Y$ is compact on $A\hatimes_\psi Y$, see \cite[Proposition~4.7]{Lance}. With $Y:=\Tt_X$, $\psi:=i_A=\varphi^{\infty}$ and
$J:=i_A^{-1}(\Kk(\Tt_X))=A$ it follows that each $\tilde\varphi^\infty(a)|_{A\hatimes_A\Tt_X}$ is compact (because $\varphi^\infty(a)|_{A}$ is compact). Use
\cite[Proposition~2.31]{tfb} to express $\xi=y\cdot\langle y,y\rangle$. We compute
\[
K_{t,\xi}=\pi_t'(T_\xi)-\pi_1'(T_\xi)=(\pi_t'(T_y)-\pi_1'(T_y))\tilde\varphi^\infty(\langle y,y\rangle),
\]
which is compact since $\tilde\varphi^\infty(\langle y,y\rangle)|_{A\hatimes_A\Tt_X}$ is compact.
As in \cite{KPS6}, $\psi_0(\xi)=\pi_0'\circ \alpha_{\Tt}(T_\xi)$ and $\psi_1(\xi)=(\pi_1'+\tau)(T_\xi)$, so we can replace $\pi_1'+\tau$ with $\pi_0'\circ \alpha_{\Tt}$ in the expression \eqref{M tensor iA is identity} for $(i_A)_*[M]-[\id_{\Tt_X}]$ without changing the class. The latter is the class of a degenerate Kasparov module, proving the claim. Thus $[M]\hatimes_A[i_A]=[\id_{\Tt_X}]$.
\end{proof}

We introduce the notation used in Lemma~\ref{Lemma: A and TtX KKequiv}. Let $I\subseteq \varphi^{-1}(\Kk(X))$ be the ideal of $A$ consisting of elements that act as compact operators on
the left of $X$. Let $\Ff_{X,I}$ denote the right
Hilbert $I$-module $\Ff_X I := \{\xi \in \Ff_X : \langle \xi, \xi\rangle_A \in I\}$. Let $\iota_I\colon I \hookrightarrow A$ and $\iota_{\Ff I}\colon\Kk(\Ff_{X,I})\hookrightarrow\Ll(\Ff_{X})$ be the
inclusion maps. As discussed in Section~\ref{sec:background}, $\iota_{\Ff I}(\Kk(\Ff_{X, I}))$ is
contained in the image of $\pi_0(\Tt_X) \subseteq \Ll(\Ff_X)$ of the Toeplitz algebra under its
canonical representation on the Fock space. Thus there is a graded embedding $j : \Kk(\Ff_{X,I})
\hookrightarrow \Tt_X$ such that $\pi_0\circ j=\iota_{\Ff I}$. We have the induced Kasparov classes
\begin{align*}
[\iota_I]&=[A_A,\iota_I,0,\alpha_A]\in KK(I,A), &
[j]&=[\Tt_X,j,0,\alpha_\Tt]\in KK(\Kk(\Ff_{X,I}),\Tt_X).
\end{align*}
Writing $P : \Ff_X \to \Ff_X \ominus A = \bigoplus^\infty_{n=1} X^{\hatimes n}$ for the
projection onto the orthogonal complement of the 0\textsuperscript{th} summand, we have
\[
[M]=\left[\Ff_X\oplus(\Ff_X\ominus A),\pi_0\oplus\pi_1\circ \alpha_{\Tt},\left(\begin{smallmatrix}
0&1\\P&0
\end{smallmatrix}\right),\left(\begin{smallmatrix}
\alpha_X^\infty&0\\0&-\alpha_X^\infty
\end{smallmatrix}\right)\right]\in KK(\Tt_X,A).
\]
We denote by $[{}_IX]$ the class $[X,\varphi|_I,0,\alpha_X]\in KK(I,A)$ of the module $X$
with the left action restricted to $I$. We note $\iota_{\Ff I}$ induces the Kasparov class
$[\Ff_{X,I},\iota_{\Ff I},0,\alpha_X^\infty]\in KK(\Kk(\Ff_{X,I}),I)$, where $\alpha_X^\infty$
and each image of $\iota_{\Ff I}$ are now considered as operators on $\Ff_{X,I}\subseteq  \Ff_{X}$.
\begin{lem}[cf. {\cite[Lemma~4.3]{KPS6}}]\label{Lemma: j tensor M}
With notation as above we have
\[
    [j]\hatimes_{\Tt_X}[M]=[\Ff_{X,I},\iota_{\Ff I},0,\alpha_X^\infty]\hatimes_I([\iota_I]-[{}_IX]).
\]
\begin{proof}
Using \cite[Proposition~18.7.2(b)]{B} we can express $[j]\hatimes_{\Tt_X}[M]$ as
\[
[j]\hatimes_{\Tt_X}[M]
    =\left[\Ff_X\oplus(\Ff_X\ominus A),
        (\pi_0\oplus(\pi_1\circ\alpha_\Tt))\circ j,
        \left(\begin{smallmatrix} 0&1\\P&0\end{smallmatrix}\right),
        \left(\begin{smallmatrix}\alpha_X^\infty&0\\0&-\alpha_X^\infty\end{smallmatrix}\right)\right].
\]
Let $\alpha_\Kk$ be the restriction of $\tilde\alpha_X^\infty$ to $\Kk(\Ff_{X,I})$. Since
$\pi_0\circ j=\iota_{\Ff I}$ and $\pi_1\circ\alpha_\Tt\circ j=\pi_1\circ j\circ\alpha_\Kk$ take values in
$\Kk(\Ff_{X,I})$, the straight line path from $\left(\begin{smallmatrix} 0&1\\P&0	
\end{smallmatrix}\right)$ to $0$ determines an operator homotopy of Kasparov modules. Hence we may
write
\begin{align}
[j]\hatimes_{\Tt_X}[M]
    &= \left[\Ff_X\oplus(\Ff_X\ominus A),(\pi_0\oplus(\pi_1\circ\alpha_\Tt))\circ j,0,\left(\begin{smallmatrix} \alpha_X^\infty&0\\0&-\alpha_X^\infty \end{smallmatrix}\right)\right]\nonumber\\
	&= [\Ff_X,\pi_0\circ j,0,\alpha_X^\infty]+[\Ff_X\ominus A, \pi_1\circ j\circ\alpha_\Kk,0,-\alpha_X^\infty]\nonumber\\
	&= [\Ff_X,\iota_{\Ff I},0,\alpha_X^\infty]-[\Ff_X\ominus A,\pi_1\circ j,0,\alpha_X^\infty]\label{eq: j tensor M}.
\end{align}
Since $\overline{\Kk(\Ff_{X,I})\Ff_X}=\Ff_{X,I}$ and $\overline{\Kk(\Ff_{X,I})(\Ff_X\ominus
A)}=\Ff_{X,I}\ominus A$ we may (see \cite[Section~17.5]{B}) replace each module in \eqref{eq: j
tensor M} with its essential submodule without changing the Kasparov classes, to obtain
\begin{equation}\label{eq: j tensor M 2}
[j]\hatimes_{\Tt_X}[M] = [\Ff_{X,I},\iota_{\Ff I},0,\alpha_X^\infty]-[\Ff_{X,I}\ominus A,\pi_1\circ j,0,\alpha_X^\infty].
\end{equation}
The map $(\xi \otimes a) \mapsto \xi \cdot a$
implements a unitary equivalence
\[
    (\Ff_{X,I}\hatimes_I A_A,\iota_{\Ff I}\hatimes 1,0,\alpha_X^\infty\hatimes \alpha_A)\cong (\Ff_{X,I},\iota_{\Ff I},0,\alpha_X^\infty).
\]
So, writing $[\Ff_{X,I}, \iota_{\Ff I}, 0, \alpha_X^\infty]_A$ for the class in $KK(\Tt_X, A)$ obtained by
regarding $\Ff_{X, I}$ as a right $A$-module, and $[\Ff_{X,I}, \iota_{\Ff I}, 0, \alpha_X^\infty]_I$ for
the class obtained by regarding it as a right $I$-module, and recalling that $[\iota_I] \in KK(I,
A)$ is the class of the inclusion of $I$ in $A$, we have
\[
[\Ff_{X,I}, \iota_{\Ff I}, 0, \alpha_X^\infty]_I \hatimes_I [\iota_I] = [\Ff_{X, I}, \iota_{\Ff I}, 0, \alpha_X^\infty]_A.
\]
Since $X$ is essential, the map that sends $i \cdot \xi$ to $i \otimes \xi$ for $i \in I$ and $\xi \in X$, and sends
$\xi_1 \otimes \cdots \otimes \xi_n$ to $(\xi_1 \otimes \cdots \otimes \xi_{n-1}) \otimes \xi_n$
for $\xi_1, \dots, \xi_n \in X$ is a unitary equivalence
\[
(\Ff_{X,I}\ominus A,\pi_1\circ j,0,\alpha_X^\infty)
    \cong (\Ff_{X,I} \hatimes_I X,(\pi_0\circ j)\hatimes 1,0,\alpha_X^\infty\hatimes\alpha_X).
\]
By \cite[Proposition~18.10.1]{B}, $[\Ff_{X,I}, \iota_{\Ff I}, 0, \alpha_X^\infty] \hatimes
[{}_IX]=[\Ff_{X,I} \hatimes_I X, \iota_{\Ff I} \hatimes 1, 0, \alpha_X^\infty\hatimes\alpha_X]$
provided that $(\iota_{\Ff I} \hatimes 1)( \Kk(\Ff_{X,I}))\subseteq \Kk(\Ff_{X,I} \hatimes_I X)$. This
containment is a direct consequence of \cite[Proposition~4.7]{Lance} since $I$ consists of
elements that act as compact operators on the left of $X$. It follows that
\[
[\Ff_{X,I}\ominus A,\pi_1\circ j,0,\alpha_X^\infty] = [\Ff_{X,I}, \iota_{\Ff I}, 0, \alpha_X^\infty] \hatimes [{}_IX].
\]
Substituting both of these equalities into \eqref{eq: j tensor M 2} and using distributivity of the Kasparov product gives
\begin{align*}
[j]\hatimes_{\Tt_X}[M]
    &=[\Ff_{X,I},\iota_{\Ff I},0,\alpha_X^\infty]-[\Ff_{X,I}\ominus A,\pi_1\circ j,0,\alpha_X^\infty]\\
    &=[\Ff_{X,I},\iota_{\Ff I},0,\alpha_X^\infty]\hatimes_I[\iota_I]+[\Ff_{X,I},\iota_{\Ff I},0,\alpha_X^\infty]\hatimes[{}_IX]\\
    &=[\Ff_{X,I},\iota_{\Ff I},0,\alpha_X^\infty]\hatimes_I([\iota_I]-[{}_IX]).\qedhere
\end{align*}
\end{proof}
\end{lem}


Our next result appears in the first author's honours thesis \cite[Theorem~7.0.3]{PattersonHons} for $I = \varphi^{-1}(\Kk(X))$, in the context where $\varphi$ is injective. For
the definition of the relative Cuntz--Pimsner algebra $\Oo_{X, I}$ see
Section~\ref{sec:background}.

\begin{thm}
Let $(A,\alpha_A),(B,\alpha_B)$ be graded separable $C^*$-algebras and suppose that $A$ is nuclear. Let $X$ be a
countably generated essential $A$--$A$-correspondence with left
action $\varphi$, and let $I \subseteq
\varphi^{-1}(\Kk(X))$ be a graded ideal of $A$. Let $\iota_I : I\to A$ be the inclusion map. Then
we have six term exact sequences
\[			
\begin{tikzpicture}[yscale=0.6, >=stealth]
	\node (00) at (0,0) {$KK_1(B, \Oo_{X,I})$};
	\node (40) at (4,0) {$KK_1(B, A)$};
	\node (80) at (8,0) {$KK_1(B, I)$};
	\node (82) at (8,2) {$KK_0(B, \Oo_{X,I})$};
	\node (42) at (4,2) {$KK_0(B, A)$};
	\node (02) at (0,2) {$KK_0(B, I)$};
	\draw[->] (02)-- node[above] {${\scriptstyle{\hatimes_A ([\iota_I] - [{_IX}])}}$} (42);
	\draw[->] (42)-- node[above] {${\scriptstyle i_*}$} (82);
	\draw[->] (82)--(80);
	\draw[->] (80)-- node[above] {${\scriptstyle{\hatimes_A ([\iota_I] - [{_IX}])}}$} (40);
	\draw[->] (40)-- node[above] {${\scriptstyle i_*}$} (00);
	\draw[->] (00)--(02);
\end{tikzpicture}
\]
and
\[
\begin{tikzpicture}[yscale=0.6, >=stealth]
	\node (00) at (0,0) {$KK_1(\Oo_{X,I}, B)$};
	\node (40) at (4,0) {$KK_1(A, B)$};
	\node (80) at (8,0) {$KK_1(I, B)$};
	\node (82) at (8,2) {$KK_0(\Oo_{X,I}, B)$};
	\node (42) at (4,2) {$KK_0(A, B)$};
	\node (02) at (0,2) {$KK_0(I, B)$};
	\draw[<-] (02)-- node[above] {${\scriptstyle{([\iota_I] - [{_IX}]) \hatimes_A}}$} (42);
	\draw[<-] (42)-- node[above] {${\scriptstyle i^*}$} (82);
	\draw[<-] (82)--(80);
	\draw[<-] (80)-- node[above] {${\scriptstyle{([\iota_I] - [{_IX}]) \hatimes_A}}$} (40);
	\draw[<-] (40)-- node[above] {${\scriptstyle i^*}$} (00);
	\draw[<-] (00)--(02);
\end{tikzpicture}
\]
\end{thm}
\begin{proof}
As in \cite[Theorem~4.4]{KPS6}, we shall prove exactness of the first diagram. Exactness of the
second follows from a similar argument. Since $A$ is nuclear, so is $\Tt_X$ by
\cite[Theorem~6.3]{RRS}. Hence the quotient map $q:\Tt_X\to\Oo_{X,I}\cong \Tt_X/j(\Kk(\Ff_{X,
I}))$ admits a completely positive splitting, see \cite[Example~19.5.2]{B}. Applying
\cite[Theorem~1.1]{S} to the graded short exact sequence
\[
\begin{tikzpicture}[yscale=0.6, >=stealth]
	\node (00) at (0,0) {$0$};
	\node (20) at (2,0) {$\Kk(\Ff_{X,I})$};
	\node (40) at (4,0) {$\Tt_X$};
	\node (60) at (6,0) {$\Oo_{X,I}$};
	\node (80) at (8,0) {$0$};
	\draw[<-] (20)--(00);
	\draw[<-] (40)-- node[above] {$j$} (20);
	\draw[<-] (60)-- node[above] {$q$} (40);
	\draw[<-] (80)--(60);
\end{tikzpicture}
\]
gives homomorphisms $\delta:KK_*(B,\Oo_{X,I})\to KK_{*+1}(B,\Kk(\Ff_{X,I}))$ for which
the following six term sequence is exact
\[
\begin{tikzpicture}[yscale=0.6, >=stealth]
	\node (00) at (0,0) {$KK_1(B, \Oo_{X,I})$};
	\node (40) at (4,0) {$KK_1(B, \Tt_X)$};
	\node (80) at (8,0) {$KK_1(B, \Kk(\Ff_{X,I}))$};
	\node (82) at (8,2) {$KK_0(B, \Oo_{X,I})$};
	\node (42) at (4,2) {$KK_0(B, \Tt_X)$};
	\node (02) at (0,2) {$KK_0(B, \Kk(\Ff_{X,I}))$};
	\draw[->] (02)-- node[above] {${\scriptstyle j_*}$} (42);
	\draw[->] (42)-- node[above] {${\scriptstyle q_*}$} (82);
	\draw[->] (82)-- node[right] {$\scriptstyle\delta$} (80);
	\draw[->] (80)-- node[above] {${\scriptstyle j_*}$} (40);
	\draw[->] (40)-- node[above] {$\scriptstyle q_*$} (00);
	\draw[->] (00)-- node[left] {$\scriptstyle\delta$} (02);
\end{tikzpicture}
\]
Define $\delta':KK_*(B,\Oo_{X,I})\to KK_{*+1}(B,A)$ by
$\delta'=(\cdot\hatimes[\Ff_{X,I},\iota_{\Ff I},0,\alpha_X^\infty])\circ\delta$ and let $i:A\to\Oo_{X,I}$
be the canonical homomorphism. Consider the following diagram.
\[
\begin{tikzpicture}[yscale=0.8]
    \node (00) at (0,0) {$KK_1(B, \Oo_{X,I})$};
    \node (40) at (4,0) {$KK_1(B, \Tt_X)$};
    \node (80) at (8,0) {$KK_1(B, \Kk(\Ff_{X,I}) )$};
    \node (82) at (8,2) {$KK_0(B, \Oo_{X,I})$};
    \node (42) at (4,2) {$KK_0(B, \Tt_X)$};
    \node (02) at (0,2) {$KK_0(B, \Kk(\Ff_{X,I}))$};
    \draw[-stealth] (02)-- node[above] {${\scriptstyle j_*}$} (42);
    \draw[-stealth] (42)-- node[above] {${\scriptstyle q_*}$} (82);
    \draw[-stealth] (82)-- node[right] {${\scriptstyle \delta}$} (80);
    \draw[-stealth] (80)-- node[below] {${\scriptstyle j_*}$} (40);
    \draw[-stealth] (40)-- node[below] {${\scriptstyle q_*}$} (00);
    \draw[-stealth] (00)-- node[left] {${\scriptstyle \delta}$} (02);
    \node (00') at (-2,-2) {$KK_1(B, \Oo_{X,I})$};
    \node (40') at (4,-2) {$KK_1(B, A)$};
    \node (80') at (10,-2) {$KK_1(B, I)$};
    \node (02') at (-2,4) {$KK_0(B, I)$};
    \node (42') at (4,4) {$KK_0(B, A)$};
    \node (82') at (10,4) {$KK_0(B, \Oo_{X,I})$};
    \draw[-stealth] (02')-- node[above] {${\scriptstyle \hatimes ([\iota_I] - [{_IX}])}$} (42');
    \draw[-stealth] (42')-- node[above] {${\scriptstyle i_*}$} (82');
    \draw[-stealth] (82')-- node[right] {${\scriptstyle \delta'}$} (80');
    \draw[-stealth] (80')-- node[below] {${\scriptstyle \hatimes ([\iota_I] - [{_IX}])}$} (40');
    \draw[-stealth] (40')-- node[below] {${\scriptstyle i_*}$} (00');
    \draw[-stealth] (00')-- node[left] {${\scriptstyle \delta'}$} (02');
    \draw[-stealth] (02)--(02') node[pos=0.25, anchor=south west, inner sep=0pt] {$\scriptstyle \hatimes [\Ff_{X,I},\iota_{\Ff I},0,\alpha_X^\infty]$};
    \draw[-stealth, out=75, in=285] (42) to node[pos=0.5, right] {${\scriptstyle\hatimes[M]}$} (42');
    \draw[-stealth, out=255, in=105] (42') to node[pos=0.5, left] {${\scriptstyle (i_A)_*}$} (42);
    \draw[-stealth] (82') to node[anchor=south east, inner sep=1pt] {$\scriptstyle\id$} (82);
    \draw[-stealth] (80)--(80') node[pos=0.25,  anchor=north east, inner sep=0pt] {$\scriptstyle \hatimes [\Ff_{X,I},\iota_{\Ff I},0,\alpha_X^\infty]$};
    \draw[-stealth, out=255, in=105] (40) to node[pos=0.5, left] {${\scriptstyle\hatimes[M]}$} (40');
    \draw[-stealth, out=75, in=285] (40') to node[pos=0.5, right] {${\scriptstyle (i_A)_*}$} (40);
    \draw[-stealth] (00') to node[anchor=north west, inner sep=1pt] {$\scriptstyle\id$} (00);
\end{tikzpicture}
\]
By definition of $\delta'$, the left and right squares commute. Lemma~\ref{Lemma: j tensor M}
shows that the top left and bottom right squares commute. By definition we have $i=q\circ i_A$ as
homomorphisms, so $i_*=(q\circ i_A)_*=q_*\circ (i_A)_*$. This shows that the top right and bottom
left squares commute.

Lemma
\ref{Lemma: A and TtX KKequiv} implies that $(i_A)_*$ and $\hatimes[M]$ are mutually inverse
isomorphisms. Finally, the class of $[\Ff_{X,I},\iota_{\Ff I},0,\alpha_X^\infty]$ is induced by the graded
Morita equivalence bimodule $\Ff_{X, I}$ (see \cite{tfb}), and so induces an isomorphism of $KK$-groups, so exactness of the interior 6-term sequence gives the desired exactness of the exterior one.
\end{proof}

\section{Graded \texorpdfstring{$K$}{K}-theory and \texorpdfstring{$K$}{K}-homology for relative Cuntz--Krieger algebras}\label{sec:graph
algebras}

In this section we use our results from the preceding section to calculate the graded $K$-theory
and graded $K$-homology of a graded relative graph $C^*$-algebra.

We first recall the key elements of relative graph $C^*$-algebras that we will use here. Given a
directed graph $E = (E^0, E^1, r, s)$, we denote by $E^0_{\rg}$ the set
\[
E^0_{\rg} := \{v \in E^0 : vE^1\text{ is finite and nonempty}\}
\]
of \emph{regular vertices} of $E$. Given a subset $V \subseteq E^0_{\rg}$, the \emph{relative
Cuntz--Krieger algebra} $C^*(E; V)$ of $E$ is defined as the universal $C^*$-algebra generated by
mutually orthogonal projections $\{p_v : v \in E^0\}$ and partial isometries $\{s_e : e \in E^1\}$
such that $s_e^*s_e = p_{s(e)}$ for all $e \in E^1$, and such that $p_v = \sum_{e \in vE^1} s_e
s^*_e$ for all $v \in V$. So $C^*(E; \emptyset)$ coincides with the usual Toeplitz algebra $\Tt
C(E)$, and $C^*(E; E^0_{\rg})$ coincides with the graph $C^*$-algebra $C^*(E)$. Given a function
$\delta : E^1 \to \{0,1\}$, the universal property of $C^*(E; V)$ yields a grading $\alpha^\delta$
of $C^*(E; V)$ satisfying $\alpha^\delta(s_e) = (-1)^{\delta(e)} s_e$ for all $e \in E^1$, and
$\alpha^\delta(p_v) = p_v$ for all $v \in E^0$.

There is a $C_0(E^0)$-valued inner-product on $C_c(E^1)$ given by
\[
\langle \xi, \eta\rangle_{C_0(E^0)}(w) = \sum_{e \in E^1 w} \overline{\xi(e)}\eta(e).
\]
The completion $X(E)$ of $C_c(E^1)$ in the resulting norm is a $C_0(E)$--$C_0(E)$-correspondence
with actions determined by $(a \cdot \xi \cdot b)(e) = a(r(e)) \xi(e) b(s(e))$. The ideal $C
\subseteq C_0(E^0)$ of elements that act by compact operators on this module is $C_0(\{v : |vE^1|
< \infty\})$. A routine argument using universal properties shows that for $V \subseteq
E^0_{\rg}$, we have $\Oo_{X(E), C_0(V)} \cong C^*(E; V)$; in particular, $\Tt_{X(E)} \cong C^*(E;
\emptyset) = \Tt C^*(E)$. Any map $\delta : E^1 \to \{0,1\}$ determines a grading $\alpha_X$ of
$X(E)$ satisfying $\alpha_X(1_e) = (-1)^{\delta(e)} 1_e$ for $e \in E^1$, and then the induced
grading on $\Oo_{X(E), C_0(V)}$ matches up with the grading $\alpha^\delta$ on $C^*(E; V)$.

Our main result for the section is the
following.

\begin{thm}\label{thm2}
Let $E$ be a directed graph. Fix a subset $V \subseteq E^0_{\rg}$ and a function $\delta : E^1 \to
\ZZ_2$. Let $\alpha \in \Aut(C^*(E; V))$ be the grading such that $\alpha(s_e) = (-1)^{\delta(e)}
s_e$ for all $e \in E^1$. Let $A^\delta_V$ denote the $V \times E^0$ matrix such that
$A^\delta_V(v,w) = \sum_{e \in vE^1 w} (-1)^{\delta(e)}$ for all $v \in V$ and $w \in E^0$.
Let $\iota : \ZZ V \to \ZZ E^0$ be the inclusion map. Regarding $(A^\delta_V)^t$ as a
homomorphism from $\ZZ V$ to $\ZZ E^0$, the graded $K$-theory of $(C^*(E; V), \alpha)$ is given by
\[
    \KTgr0(C^*(E; V), \alpha) \cong \coker(\iota - (A^\delta_V))^t\quad\text{ and }\quad
    \KTgr1(C^*(E; V), \alpha) \cong \ker(\iota - (A^\delta_V))^t.
\]
There is a homomorphism $(\widetilde{A}^\delta_V) : \ZZ^{E^0}\to \ZZ^{V}$
given by
\[\textstyle
(\widetilde{A}^\delta_V)(f)(v) = \sum_{e \in vE^1} A^\delta_V(v,w) f(w)\quad\text{ for all $v \in V$ and $f \in \ZZ^{E^0}$.}
\]
Let $\pi : \ZZ^{E^0} \to \ZZ^V$ be the projection map. Then
\[
    \KHgr0(C^*(E; V), \alpha) \cong \ker(\pi - \widetilde{A}^\delta_V)\quad\text{ and }\quad
    \KHgr1(C^*(E; V), \alpha) \cong \coker(\pi - \widetilde{A}^\delta_V).
\]
\end{thm}

Before proving the theorem, we state an immediate corollary about the graded $K$-theory and
$K$-homology of graph $C^*$-algebras.

\begin{cor}\label{cor.graph}
Let $E$ be a directed graph. Fix a function $\delta : E^1 \to \ZZ_2$, and let $\alpha$ be the
associated grading of $C^*(E)$. Let $A^\delta_{\rg}$ denote the $E^0_{\rg} \times E^0$ matrix such
that $A^\delta_{\rg}(v,w) = \sum_{e \in vE^1 w} (-1)^{\delta(e)}$ for all $v \in V$ and $w \in
E^0$ regarded as a homomorphism from $\ZZ E^0_{\rg}$ to $\ZZ E^0$ and write
$\widetilde{A}^\delta_{\rg}$ for the dual homomorphism $\ZZ^{E^0} \to \ZZ^{E^0_{\rg}}$. Then
\begin{align*}
    \KTgr0(C^*(E), \alpha) &\cong \coker(\iota - (A^\delta_{\rg}))^t,\qquad &\KTgr1(C^*(E), \alpha) &\cong \ker(\iota - (A^\delta_{\rg}))^t,\\
    \KHgr0(C^*(E), \alpha) &\cong \ker(\pi - \widetilde{A}^\delta_{\rg}),\quad\text{ and}\qquad &\KHgr1(C^*(E), \alpha) &\cong \coker(\pi - \widetilde{A}^\delta_{\rg}).
\end{align*}
\end{cor}
\begin{proof}
Apply Theorem~\ref{thm2} to $V = E^0_{\rg}$.
\end{proof}

To prove the $K$-homology formulas in the theorem, we need some preliminary results; the
corresponding results for $K$-theory are established in \cite{KPS6}.

\begin{prp}\label{Prop: Khom direct sum iso}
Let $E$ be a row-finite directed graph. For $\nn\in \ZZ^{E^0}$ and $f\in C_0(E^0)$ let
$\ell(f)\in\Ll\big(\bigoplus_{v\in E^0}\CC^{|\nn_v|}\big)$ be given by
\[
       (\ell(f)x)_w=f(w)x_w, \ \ \textrm{for each} \ w\in E^0.
\]
Then there
is an isomorphism $\mu:\ZZ^{E^0}\to KK_0(C_0(E^0),\CC)$ such that
\[
\mu(\nn) = \Big[\bigoplus_{v\in E^0}\CC^{|\nn_v|},\ell,0,\bigoplus_{v\in E^0}\sign(\nn_v)\Big].
\]
\end{prp}
\begin{proof}
Since $E^0$ is discrete we can identify $C_0(E^0)$ with $\bigoplus_{v\in E^0}\CC$.	For each $w\in E^0$, let
$g_w:\CC\to\bigoplus_{v\in E^0}\CC$ be the coordinate inclusion. Then Theorem~19.7.1 of \cite{B} implies that $\theta:=\prod_{v\in
E^0}g_w^*:KK(\bigoplus_{v\in E^0}\CC,\CC)\to\prod_{v\in E^0}KK(\CC,\CC)\cong\ZZ^{E^0}$ is an
isomorphism of groups, with inverse $\mu$.
\end{proof}

\begin{lem}\label{Lemma: techLemUnitaryEquiv}
Let $E$ be a directed graph and let $\delta : E^1 \to \{0,1\}$ be a function. Fix a subset $V \subseteq E^0_{\rg}$. Fix $v\in E^0$. For $f\in E^1v$
and $a\in C_0(E^0)$, define $\psi^v(a):\ell^2(E^1v)\to\ell^2(E^1v)$ on elementary basis vectors
$\{e^f\}\subseteq\ell^2(E^1v)$ by
\[
    \psi^v(a)e^f =
        \begin{cases}
            a(r(f))e^f &\text{ if $r(f) \in V$}\\
            0 &\text{ otherwise.}
        \end{cases}
\]
Define $\beta:\ell^2(E^1v)\to\ell^2(E^1v)$ on basis elements by
\[
\beta(e^f)=(-1)^{\delta(f)}e^f.
\]
Let $\phi_V : C_0(V) \to \Kk(X(E))$ be the restriction of
the left action. Then $(\ell^2(E^1v),\psi^v,0,\beta)$ is a Kasparov $C_0(E^0)$--$\CC$-module, and
for each $\nn_v\in\ZZ$ there is a unitary equivalence between the modules
\[
(\ell^2(E^1v)\hatimes\CC^{|\nn_v|},\tilde{\psi},0,\beta\hatimes\sign(\nn_v)\id)
\]
and
\[
(X(E)\hatimes_{\eps_v}\CC^{|\nn_v|},\tilde\phi_V,0,\alpha_X\hatimes\sign(\nn_v)\id).
\]
\end{lem}
\begin{proof}
Throughout the proof, we write $\psi$ for $\psi^v$. For each $w \in V$, the operator
$\psi(\delta_w) = \sum_{f \in wE^1} \Theta_{e^f, e^f}$ is compact, and for $w \in E^0 \setminus
V$, we have $\psi(\delta_w) = 0$. So each $\psi(a) = \bigoplus_{w \in V} a(w)\psi(\delta_w)$ is
compact.

It is immediate that $\beta$ is a grading, and it preserves the grading of the left action since
$C_0(E^0)$ carries the trivial grading. So $(\ell^2(E^1v),\psi,0,\beta)$ is a Kasparov
$C_0(V)$--$\CC$-module.

Recall that for an edge $f\in E^1v$ the element $\delta_f\in X(E)$ denotes the point mass at $f$.
Further for $j\leq|\nn_v|$ let $e^j$ be an orthonormal basis for $\CC^{|\nn_v|}$. We claim there
is a unitary equivalence $U:\ell^2(E^1v)\hatimes\CC^{|\nn_v|}\to
X(E)\hatimes_{\eps_v}\CC^{|\nn_v|}$ that acts on elementary basis tensors $e^f\hatimes e^j$ by the
formula
\[
U(e^f\hatimes e^j)=\delta_f\hatimes e^j.
\]
An elementary calculation shows that this formula preserves inner-products, so extends to a well
defined isometry. To see that $U$ is surjective, note that any function $x\in X(E)$ that is zero
on $E^1v$ satisfies $x\hatimes w=0$ for any $w\in\CC^{|\nn_v|}$. Thus it suffices to consider
functions contained in the span of $\{\delta_f:f\in E^1v\}=\overline{C_c(E^1v)}$, and tensors
$\delta^f\hatimes e^j$ can be written in the form $U(e^f\hatimes e^j)$ by construction.

The definitions of $\phi_V$ and $\psi^v$ show that $U$ intertwines the left actions. To see that
it preserves gradings, fix $f\in E^1$ and $j\leq|\nn_v|$ and compute
\begin{align*}
U\big((\beta\hatimes\sign(\nn_v)\id)(e^f\otimes e^j)\big)&=\sign(\nn_v)(-1)^{\delta(f)}U(e^f\otimes e^j)\\
&=\sign(\nn_v)(-1)^{\delta(f)}\delta^f\hatimes e^j\\
&=(\alpha_X\hatimes\sign(\nn_v)\id)(\delta^f\hatimes e^j).\qedhere
\end{align*}
\end{proof}

\begin{prp}\label{prp: Jono thesis main}
Let $E$ be a directed graph. Fix a subset $V \subseteq E^0_{\rg}$ and a function $\delta : E^1 \to
\{0,1\}$. Let $A^\delta_V$ denote the $V \times E^0$ matrix such that $A^\delta_V(v,w) =
\sum_{e \in vE^1 w} (-1)^{\delta(e)}$ for all $v \in V$ and $w \in E^0$. Let $[X_{E, V}] =
[X(E),\phi|_{C_0(V)},0,\alpha_X]$ be the Kasparov class of the graded graph module with left
action restricted to $C_0(V)$. Let $\mu$ denote both the isomorphism $KK_0(C_0(E^0), \CC) \to
\ZZ^{E^0}$ and the isomorphism $KK_0(C_0(V), \CC) \to \ZZ^V$ from Proposition~\ref{Prop: Khom
direct sum iso}. Let $[\iota] \in KK(J_X, A)$ be the class of the inclusion, and let $\pi :
\ZZ^{E^0} \to \ZZ^{E^0_{\rg}}$ be the projection map. Then the following diagrams commute.
\[
	\begin{tikzpicture}[yscale=0.45, >=stealth]
	\node (00) at (0,0) {$KK_0(C_0(V), \CC)$};
	\node (40) at (4.25,0) {$KK_0(C_0(E^0), \CC)$};
	\node (42) at (4.25,3) {$\ZZ^{E^0}$};
	\node (02) at (0,3) {$\ZZ^V$};
	\draw[->] (42)-- node[above] {${\scriptstyle{A^\delta_V}}$} (02);
	\draw[->] (40)-- node[below] {${\scriptstyle{[X_{E, V}]\hatimes}}$} (00);
	\draw[->] (00)-- node[left] {${\scriptstyle{\mu}}$} (02);
	\draw[->] (40)-- node[right] {${\scriptstyle{\mu}}$} (42);
	\node (70) at (8.25,0) {$KK_0(C_0(V), \CC)$};
	\node (120) at (12.5,0) {$KK_0(C_0(E^0), \CC)$};
	\node (122) at (12.5,3) {$\ZZ^{E^0}$};
	\node (72) at (8.25,3) {$\ZZ^V$};
	\draw[->] (122)-- node[above] {$\scriptstyle\pi$} (72);
	\draw[->] (120)-- node[below] {${\scriptstyle{[\iota]\hatimes}}$} (70);
	\draw[->] (70)-- node[left] {${\scriptstyle{\mu}}$} (72);
	\draw[->] (120)-- node[right] {${\scriptstyle{\mu}}$} (122);
	\end{tikzpicture}
\]
\end{prp}
\begin{proof}
By surjectivity of the isomorphism $\mu$ each element of $KK_0(C_0(E^0),\CC)$ is of the form
\[
\mu(\nn)=\Big[\bigoplus_{v\in E^0}\CC^{|\nn_v|},\ell,0,\bigoplus_{v\in E^0}\sign(\nn_v)\Big]
\]
for some $\nn\in\ZZ^{E^0}$. In what follows, we write $\phi_V$ for $\phi|_{C_0(V)}$. We compute
using Lemma~\ref{Lemma: techLemUnitaryEquiv}
\begin{align*}
	[X(E),\phi_V,0,\alpha_X]\hatimes\mu(\nn)&=\Big[\bigoplus_{v\in E^0}(X(E)\hatimes\CC^{|\nn_v|}),\bigoplus_{v\in E^0}\phi\hatimes 1,0,\bigoplus_{v\in E^0}(\alpha_X\hatimes\sign(\nn_v))\Big]\\
	&=\Big[\bigoplus_{v\in E^0}\ell^2(E^1v)\hatimes\CC^{|\nn_v|},\bigoplus_{v\in E^0}\psi^v\hatimes 1,0,\bigoplus_{v\in E^0}(\beta\hatimes\sign(\nn_v))\Big].
\end{align*}
Let $g_v:\CC\to C_0(E^0)$ be the coordinate inclusion corresponding to the vertex $v\in E^0$.
Theorem~19.7.1 of \cite{B} gives an isomorphism $\theta:=\prod_{v\in
E^0}g_v^*:KK(C_0(E^0),\CC)\to\prod_{v\in E^0}KK(\CC,\CC)\cong \ZZ^{E^0}$. Applying $\theta$ to
$[X_{E, V}]\hatimes\mu(\nn)$ gives
\begin{align}
	\theta\Big(\Big[\bigoplus_{v\in E^0}&\ell^2(E^1v)\hatimes\CC^{|\nn_v|},\bigoplus_{v\in E^0}\psi^v\hatimes 1,0,\bigoplus_{v\in E^0}(\beta\hatimes\sign(\nn_v))\Big]\Big)\nonumber\\
	&=\Big(\Big[\bigoplus_{v\in E^0}\ell^2(E^1v)\hatimes\CC^{|\nn_v|},\Big(\bigoplus_{v\in E^0}\psi^v\hatimes 1\Big)\circ g_w,0,\bigoplus_{v\in E^0}(\beta\hatimes\sign(\nn_v))\Big]_{w}\Big)_{w\in E^0}.\label{eq: thetaModule}
\end{align}
Since the action $(\psi\hatimes 1)\circ g_w$ is zero on the submodule $\ell^2(E^1v\setminus
wE^1v)\hatimes \CC^{|\nn_v|}$ for each $v,w\in E^0$, Equation~\eqref{eq: thetaModule} becomes
\begin{align}
	\Big(\Big[\bigoplus_{v\in E^0}&\ell^2(E^1v)\hatimes\CC^{|\nn_v|},\Big(\bigoplus_{v\in E^0}\psi^v\hatimes 1\Big)\circ g_w,0,\bigoplus_{v\in E^0}(\beta\hatimes\sign(\nn_v))\Big]_{w}\Big)_{w\in E^0}\nonumber\\
	&=\Big(\Big[\bigoplus_{v\in E^0}\ell^2(wE^1v)\hatimes\CC^{|\nn_v|},\Big(\bigoplus_{v\in E^0}\psi^v\hatimes 1\Big)\circ g_w,0,\bigoplus_{v\in E^0}(\beta\hatimes\sign(\nn_v))\Big]_{w}\Big)_{w\in E^0}\nonumber\\
	&=\Big(\Big[\bigoplus_{v\in E^0}\CC^{|A^\delta(w,v)\nn_v|},\Big(\bigoplus_{v\in E^0}\psi^v\hatimes 1\Big)\circ g_w,0,\bigoplus_{v\in E^0}(\sign(A^\delta(w,v)\nn_v))\Big]_{w}\Big)_{w\in E^0}.\label{eq: finiteOplus}
\end{align}
Since $\psi^v(\delta_w)$ is zero when $w \not\in V$, we may pass to the essential submodule, so
that~\eqref{eq: finiteOplus} becomes
\begin{align*}
	\Big(\Big[\bigoplus_{v\in E^0}&\CC^{|A^\delta(w,v)\nn_v|},\Big(\bigoplus_{v\in E^0}\psi\hatimes 1\Big)\circ g_w,0,\bigoplus_{v\in E^0}(\sign(A^\delta(w,v)\nn_v))\Big]_{w}\Big)_{w\in V}\\
	&=\Big(\Big[\CC^{|\sum_{v\in E^0}A^\delta(w,v)\nn_v|},\cdot, 0,\sign\Big(\sum_{v\in E^0}A^\delta(w,v)\nn_v\Big)\Big]_{w}\Big)_{w\in V}.
\end{align*}
This is exactly the representative of $A^\delta_V\nn \in \ZZ^{E^0}$ as a module in
$KK_0(\CC,\CC)^V$. Hence $[X_{E, V}]\hatimes\mu(\nn)=\mu(A^\delta_V\nn)$ as required.

The final statement follows directly from the definition of $\mu$.
\end{proof}

\begin{proof}[Proof of Theorem~\ref{thm2}]
For the bimodule $X(E)$, the coefficient algebra $A$ is $C_0(E^0) = \bigoplus_{v \in E^0} \CC$ and
the ideal $I \subseteq C$ is $C_0(V) = \bigoplus_{v \in V} \CC$. Countable additivity of
$K$-theory (or Proposition~\ref{prp:KK(C,A) countable additivity}) shows that $KK_0(\CC, C_0(E^0)) \cong \ZZ E^0$ and $KK_0(\CC, I) \cong \ZZ V$. The
argument of \cite[Lemma~8.2]{KPS6} shows that these isomorphisms intertwine the map $\cdot
\hatimes [{_I X(E)}]$ from $KK_0(\CC, I)$ to $KK(\CC, A)$ with the map $(A^\delta_V)^t : \ZZ V \to
\ZZ E^0$. Functoriality of $KK(\CC, \cdot)$ shows that these isomorphisms also intertwine
$[\iota]$ with the inclusion map $\iota : \ZZ V \to \ZZ E^0$. So the exact sequence of
Theorem~\ref{Theorem: First half of main thm A} induces the exact sequence
\[
0 \to \KTgr{1}(C^*(E; V), \alpha) \to \ZZ V \xrightarrow{1 - (A^\delta_V)^t} \ZZ E^0 \to \KTgr{0}(C^*(E; V), \alpha) \to 0,
\]
and the first part of the theorem follows.

For the second statement, first note that Proposition~\ref{Prop: Khom direct sum iso} and
Theorem~19.7.1 of \cite{B} give isomorphisms $KK_0(A, \CC) \cong \ZZ^{E^0}$ and
$KK_0(I, \CC) \cong \ZZ^V$, and show that $KK_1(A, \CC) = KK(J_X, \CC) = \{0\}$.
Proposition~\ref{prp: Jono thesis main} shows that these isomorphisms intertwine $(1 - [{_IX(E)}])
\hatimes \cdot$ with $(1 - \tilde{A}^\delta_V)$. So the exact sequence of Theorem~\ref{Theorem:
Second half of main thm A} gives the exact sequence
\[
0 \to \KHgr{0}(C^*(E; V), \alpha) \to \ZZ^{E^0} \xrightarrow{1 - A^\delta_{V}} \ZZ^V \to \KHgr{1}(C^*(E; V), \alpha) \to 0.\qedhere
\]
\end{proof}

\section{Adding tails to graded correspondences}\label{sec:adding tails}

Inspired by the technique of `adding tails' to directed graphs which transforms a directed graph
into a graph without sources with a Morita equivalent $C^*$-algebra, Muhly and Tomforde proved
that given an $A$--$A$ correspondence $X$ with non-injective left action, there is a
correspondence $Y$ over a $C^*$-algebra $B$ such that the left action of $B$ on $Y$ is implemented
by an injective homomorphism, and $\Oo_X$ is a full corner of $\Oo_Y$. If $A$ and $X$ are graded,
these gradings extend naturally to gradings on $B$ and $Y$, and the inclusion of $\Oo_X$ in
$\Oo_Y$ is graded. In particular $(\Oo_X, \alpha_X)$ and $(\Oo_Y, \alpha_Y)$ are $KK$-equivalent,
and in particular have isomorphic graded $K$-theory and $K$-homology.


In this section we show how adding tails to a correspondence has applications in $KK$-theory. More
specifically we show how to recover Pimsner's 6-term exact sequences for $KK(\cdot, B)$ and
$KK(\CC, \cdot)$ for any graded countably generated essential $A$--$A$-correspondence $X$ using
\emph{only} the corresponding sequences for graded correspondences with an injective left action
(these were established in the first author's honours thesis \cite{PattersonHons}). This yields
the calculations of graded $K$-theory and $K$-homology of $C^*$-algebras of arbitrary graphs that
first appeared in the first author's honours thesis \cite{PattersonHons} and the second author's
honours thesis \cite{TaylorHons} respectively (see Corollary~\ref{cor.graph}). It also provides a
useful reality check for our more general results in the preceding sections; we thank Ralf Meyer
for pointing us to the direct proof of Pimsner's sequences for modules with non-injective left
actions employed there. To exploit the adding-tails technique, we need to know that graded
$K$-theory and $K$-homology are each countably additive in the appropriate sense. We provide
details below.

\subsection{Countable additivity}\label{sec:additive}

In what follows, a direct sum of graded $C^*$-algebras is endowed with the natural direct-sum
grading; so given a Kasparov class $[X] = [X, \phi, F, \alpha_X] \in KK(\bigoplus_i A_i, B)$, for each $n$, the
inclusion map $\iota_{A_n} : A_n \hookrightarrow \bigoplus_i A_i$ is graded and induces the class
$\iota_{A_n}^*[X]\in KK(A_n,B)$.

Theorem~19.7.1 of \cite{B} shows that for separable $B$, $KK(\,\cdot\,, B)$ is countably additive
in the following sense. If $A=\bigoplus_{n=1}^\infty A_n$ is a $c_0$-direct sum of separable
$C^*$-algebras $A_n$, then there is an isomorphism $\overline{\zeta} : KK(A, B) \to
\prod^\infty_{n=1} KK(A_n, B)$ that carries the class of a Kasparov $A$--$B$-module $(X, \phi, F,
\alpha_X)$ to the sequence whose $n$th term is the class of $(X, \phi\circ\iota_{A_n}, F,
\alpha_X)$; that is, in the notation of Section~\ref{sec:KKBasics}, writing $[X]$ for the class of
$(X, \phi, F, \alpha_X)$, we have
\[
\overline{\zeta}[X] = \big(\iota_{A_n}^*[X]\big)^\infty_{n=1}.
\]
In particular, identifying $KK(A,\CC)$ with $\KHgr0(A)$, we obtain countable additivity of graded
$K$-homology.

On the other hand,
as discussed in \cite[19.7.2]{B}, the map $KK(B, \,\cdot\,)$  is typically not countably additive. There is a natural map $\omega$ from $\bigoplus^\infty_{i=1} KK(A, B_i)$
to $KK(A, \bigoplus^\infty_{i=1} B_i)$ such that
\begin{align}
\label{omegamap}
\textstyle \omega(([X_i, \phi_i, F_i, \alpha_i])^n_{i=1})=[\bigoplus^n_{i=1} X_i, \bigoplus^n_{i=1} \phi_i, \bigoplus^n_{i=1} F_i, \bigoplus^n_{i=1}
\alpha_i].
\end{align}
This $\omega$ is always injective (see the proof of Proposition~\ref{prp:KK(C,A) countable additivity} below), but is typically not surjective.
Since
our focus is on graded $K$-theory and graded $K$-homology, we content ourselves with recording the
presumably well-known fact that $KK(\CC, \,\cdot\,)$ is countably additive. For ungraded
$C^*$-algebras, this follows from countable
additivity of $K$-theory since $KK(\CC,A)\cong K_0(A)$.

\begin{prp}\label{prp:KK(C,A) countable additivity}
Let $(A_i, \alpha_i)^\infty_{i=1}$ be a sequence of $\sigma$-unital graded $C^*$-algebras. Then
the map $\omega : \bigoplus^\infty_{i=1} KK(\CC, A_i) \to KK(\CC, \bigoplus^\infty_{i=1} A_i)$ defined at \eqref{omegamap} is
an isomorphism.
\end{prp}
\begin{proof}
Write $A_\infty := \bigoplus_i A_i$. Let $(X, \phi, F, \alpha_X)$ be a Kasparov
$\CC$--$A_\infty$-module. Since each $A_i$ is an ideal of $A_\infty$, each $X_i := \{\xi : \langle
\xi, \xi\rangle \in A_i\}$ is a right-Hilbert $A_i$-module. We identify $X$ and $\bigoplus_i X_{i}$ as
right-Hilbert $A_\infty$-modules, so $X=\clsp\{x_i\in X_i: i\geq 1\}$ and $\langle x_i,
x_j\rangle_A = 0$ for all $i\neq j$. Since  $F$ and $\phi(1)$ are adjointable, they leave each
$X_i$ invariant ($FX_i\subseteq X_i$ and $\phi(1)X_i\subseteq X_i$); so $F$ and $\phi$ decompose
as $F = \bigoplus F_i$ and $\phi = \bigoplus \phi_i$. Since the inclusion map
$\iota_{A_i}:A_i\hookrightarrow A_\infty$ is graded and $X_i =X\cdot A_i$, the automorphism
$\alpha_X$ also leaves each $X_i$ invariant; so $\alpha_X$ decomposes as a direct sum $\alpha_X = \bigoplus
\alpha_i$. Since each compact operator on $X=\bigoplus_i X_i$ restricts to a compact operator on
$X_i$, each $(X_i, \phi_i, F_i, \alpha_i)$ is a Kasparov $\CC$--$A_i$-module.

Applying the preceding paragraph with $A_\infty$ replaced by $C([0,1], A_\infty)$ shows that any
homotopy of Kasparov $\CC$--$A_\infty$ modules decomposes as a direct sum of homotopies of
Kasparov $\CC$--$A_i$-modules, and so $\omega$ is injective.

To show that $\omega$ is surjective, we claim that it suffices to show that if $(X, \phi, F,
\alpha)$ is a Kasparov $\CC$--$A_\infty$-module, then there exists $N$ such that $[X_i, \phi_i,
F_i, \alpha_i] = 0_{KK(\CC, A_i)}$ for all $i \ge N$. To see why, suppose that $[X_i, \phi_i, F_i,
\alpha_i] = 0_{KK(\CC, A_i)}$ for all $i \ge N$. Let
\[
[X^\infty_N] := \Big[0^{N-1} \oplus \bigoplus^\infty_{i=N} X_i, 0^{N-1} \oplus \bigoplus^\infty_{i=N} \phi_i, 0^{N-1} \oplus \bigoplus^\infty_{i=N} F_i, 0^{N-1} \oplus \bigoplus^\infty_{i=N} \alpha_i\Big]
    \in KK(\CC, A_\infty).
\]
We have
\[ \textstyle
    [X, \phi, F, \alpha] = \omega\big(\bigoplus^{N-1}_{i=1}[X_i, \phi_i, F_i, \alpha_i]\big) + [X^\infty_N].
\]
Since, for each $i \ge N$, the module $(X_i, \phi_i, F_i, \alpha_i)$ is a degenerate Kasparov
module, so is $\big(\bigoplus^\infty_{i=N} X_i, \bigoplus^\infty_{i=N} \phi_i,
\bigoplus^\infty_{i=N} F_i, \bigoplus^\infty_{i=N} \alpha_i\big)$. So $[X^\infty_N] = 0_{KK(\CC,
A_\infty)}$, and it follows that $[X, \phi, F, \alpha] = \omega(\bigoplus^{N-1}_{i=1}[X_i, \phi_i,
F_i, \alpha_i])$ belongs to the range of $\omega$.

To see that there exists $N$ such that $[X_i, \phi_i, F_i, \alpha_i] = 0_{KK(\CC, A_i)}$ for all
$i \ge N$, first note that by \cite[Propositions 17.4.2~and~18.3.6]{B}, we may assume that
$\phi(1) = 1_X$,  that $F = F^*$ and that $\|F\| \le 1$. Using that $\Kk(X)=\clsp\{\Theta_{\xi,
\eta}: \xi, \eta\in \bigoplus_{i=1}^n X_i, n\in \NN\}$ it follows that if $T \in \Kk(X)$, then
$\|T|_{X_i}\| \to 0$. Since $\phi_i(1)=1_{X_i}$ and $F^2 - 1 = (F^2 - 1)\phi(1) \in \Kk(X)$, we
deduce that $\|F_i^2 - 1\| \to 0$. So there exists $N$ large enough so that $F_i^2$ is invertible
for $i \ge N$. Fix $i \ge N$; we will show that $(X_i, \phi_i, F_i, \alpha_i)$ is degenerate.
Since $F_i^* = F_i$ we see that $F_i$ is normal, and so $\sigma(F_i^2) = \sigma(F_i)^2$, and in
particular, $F_i$ is invertible. Since $\|F_i\| \le 1$, we have $\sigma(F_i) \subseteq [-1,0) \cup
(0,1]$, so for $t \in [0,1]$ there is a continuous function $f_t \in C(\sigma(F_i))$ given by
\[
f_t(x) = \begin{cases}
        (1-t)x + t &\text{ if $x > 0$}\\
        (1-t)x - t &\text{ if $x < 0$.}
    \end{cases}
\]
Now the path $(F_t)_{t \in [0,1]}$ is a continous path from $F_i$ to $f_1(F_i)$. Since $\sigma(f_1(F_i)))=f_1(\sigma(F_i))=\{-1,1\}$, we have $f_1(F_i)^2 = 1$.

We claim that for each $t$, the tuple $(X_i, \phi_i, f_t(F_i), \alpha_i)$ is a Kasparov module. To
see this, first note that each $f_t(F_i)^* = f_t(F_i)$. Since $\phi_i(\CC) = \lsp 1_{X_i}$ is
even-graded and central, we have $[\phi_i(a), f_t(F_i)]^{\operatorname{gr}} = a[\phi_i(1),
f_t(F_i)] = 0$ for all $t,a$. Since $F_i$ is odd with respect to the grading on $\Ll(X_i)$, so is
$F_i^{2n+1}$ for every $n \ge 0$. So writing $P_{\operatorname{odd}}$ for the space $\big\{z
\mapsto \sum^N_{n=0} a_n z^{2n+1} \mid N \ge 0\text{ and }a_i \in \CC\big\}$ of odd polynomials,
$f(F_i)$ is odd for each $f \in P_{\operatorname{odd}}$. Since the $f_t$ are all odd
functions, they can be uniformly approximated by elements of $P_{\operatorname{odd}}$, and we
deduce that each $f_t(F_i)$ is odd with respect to $\tilde\alpha_i$. So to prove the claim, it
remains to show that each $f_t(F_i)^2-1 \in \Kk(X_i)$. For this, observe that the
functional-calculus isomorphism for $F_i$ carries $F_i^2 - 1$ to the function $z \mapsto z^2 - 1$.
Since this function vanishes only at $1$ and $-1$, the ideal of $C^*(F_i)$ generated by $F_i^2 -
1$ is $\{f(F_i) : f(1) = f(-1) = 0, f\in C(\sigma(F_i))\}$. Since each $f_t^2(1) = 1 = f_t^2(-1)$, we deduce that each
$f_t(F_i)^2 - 1$ belongs to the ideal generated by $F_i^2 - 1$, and since $F_i^2 - 1 \in \Kk(X_i)$
it follows that each $f_t(F_i)^2 - 1 \in \Kk(X_i)$. This proves the claim.

Hence $(X_i, \phi_i, f_t(F_i), \alpha_i)_{t \in [0,1]}$ is an operator homotopy, and it follows
that
\[
[X_i, \phi_i, F_i, \alpha_i] = [X_i, \phi_i, f_1(F_i), \alpha_i].
\]
By construction, $f_1(F_i)^2 = 1$, and we already saw that $f_1(F_i)$ is odd, self-adjoint and
commutes with $\phi_i(1)$. So $[X_i, \phi_i, f_1(F_i), \alpha_i] = 0_{KK(\CC, A_i)}$ as required.
\end{proof}

\textbf{Set-up.} \emph{Throughout the remainder of this section we fix a graded, separable,
nuclear $C^*$-algebra $A$ and a graded countably generated essential $A$--$A$-correspondence $X$
with left action implemented by $\varphi : A \to \Ll(X)$.}

\smallskip
Recall that if $I$ is an ideal of $A$, then $I^\perp$ is the ideal $\{b \in A : b I = \{0\}\}$.
Following Katsura, we define $J_X := \varphi^{-1}(\Kk(X))\cap\ker\varphi^\perp \lhd A$. Since $ja
= aj = 0$ for all $j \in J_X$ and $a \in \ker\varphi$, the ideal $J_X + \ker\varphi \subseteq A$
is the internal direct sum $J_X \oplus \ker\varphi$. We sometimes identify it with the external
direct sum via the map $j + a \mapsto (j,a)$.

Since $J_X$ acts compactly on $X$, the quadruple $(X, \varphi|_{J_X}, 0, \alpha_X)$ is a Kasparov
module, and we write $[X]$ for the corresponding class in $KK(J_X, A)$.

We define $K_\varphi^\infty := \bigoplus_{n=1}^\infty \ker\varphi$ as a graded $C^*$-algebra, and $T :=
(K_\varphi^\infty)_{K_\varphi^\infty}$ regarded as an $(A \oplus
K_\varphi^\infty)$--$K_\varphi^\infty$-correspondence with left action $(a,f) \cdot g = (ag_1, f_1
g_2, f_2 g_3, \dots)$.

We define $Y:=X\oplus T$ as a right-Hilbert $A \oplus K_\varphi^\infty$-module, so the right
action of $A \oplus K_\varphi^\infty$ is given by
\[
    (x,f)\cdot (a,g)=(x\cdot a,fg)
\]
and the inner product is given by
\[
\langle (x,f),(y,g)\rangle=(\langle x,y\rangle, f^*g),
\]
for $x,y\in X$, $a\in A$, and $f,g\in T$.

Viewing the left action of $A$ on $X$ as an action of $A \oplus K_\varphi^\infty$ in which the
second coordinate acts trivially, $Y$ is an $(A \oplus K_\varphi^\infty)$--$(A \oplus
K_\varphi^\infty)$-correspondence with left action
\[
    \varphiY(a,f)(x,g) = (\varphi(a)x,ag_1,f_1g_2,\cdots, f_ng_{n+1},\dots).
\]
The homomorphism $\varphiY$ is injective (see \cite[Lemma~4.2]{MT}).

\begin{prp}
\label{ideal.I}
The ideal $(\varphiY)^{-1}(\Kk(Y))$ is equal to $J_X \oplus \ker\varphi \oplus
K_\varphi^\infty \subseteq A \oplus K_\varphi^\infty$.
\end{prp}
\begin{proof}
As discussed just before the statement of the lemma, we have $J_X + \ker\varphi = J_X \oplus
\ker\varphi$, so it suffices to show that $(\varphiY)^{-1}(\Kk(Y)) = (J_X +
\ker\varphi) \oplus K_\varphi^\infty$

To prove that $(J_X + \ker\varphi) \oplus K_\varphi^\infty \subseteq (\varphiY)^{-1}(\Kk(Y))$, fix $j+a\in J_X + \ker\varphi$, and $f\in K_\varphi^\infty$. For $x\in X,g\in T$, since
$\varphi(a + j) = \varphi(j)$ and since $j \cdot T = 0$, we have
\[
    \varphiY(j+a,f)(x,g)=(\varphi(j)x,ag_1,f_1g_2,f_2g_3,\dots).
\]
Since $j\in J_X$ we have $\varphi(j)\in\Kk(X)$. Further, letting $F=(a,f_1,f_2,\dots)$ and letting
$L_F\in\Kk(T)$ denote the left multiplication operator by $F$ we have
\[
\varphiY(j+a,f)(x,g)=(\varphi(j)x,a g_1,f_1g_2,\dots)=(\varphi(j)x,0,0,\dots)+(0,L_Fg),
\]
so $\varphiY(j+a,f)=(\varphi(j),L_F)\in\Kk(X)\oplus\Kk(T)\subset\Kk(X\oplus T)$.

To prove that $(\varphiY)^{-1}(\Kk(Y)) \subseteq J_X \oplus \ker\varphi \oplus
K_\varphi^\infty$, first note that $\varphiY$ decomposes as an (internal) direct sum
$\varphi^A \oplus \varphi^T$ of the homomorphisms $\varphi^A$ and $\varphi^T$ given by
\[
\varphi^A(a,f)(x,g_1,g_2,\dots)=(\varphi(a)x,ag_1,0,0\dots),\,\text{and }\varphi^T(a,f)(x,g)=(0,0,f_1g_2,f_2g_3,\dots).
\]

Suppose that $(a,f)\in(\varphiY)^{-1}(\Kk(Y))$; that is, $\varphiY(a,f)$ is
compact. Since $\varphi^T(a,f)$ is left multiplication by $(0,0,f_1,f_2,\dots)$, it is compact.
Hence $\varphi^A=\varphiY-\varphi^T$ is also compact. In particular $\varphi(a)$ is
compact as it is the restriction of $\varphi^A(a,f)$ to the first entry, and so $a \in
\varphi^{-1}(\Kk(X))$. Since $a$ also acts compactly on the first coordinate of $T$, which is the
right-Hilbert module $(\ker\varphi)_{\ker\varphi}$, we see that the left-multiplication-by-$a$
operator on $(\ker\varphi)_{\ker\varphi}$ agrees with left-multiplication by some $a' \in
\ker\varphi$. In particular, $j := a - a' \in (\ker\varphi)^\perp$. Since $\varphi(j) = \varphi(a)
- \varphi(a') = \varphi(a)$, we have $j \in \varphi^{-1}(\Kk(X)) \cap (\ker\varphi)^\perp = J_X$.
Hence $a = j + a' \in J_X + \ker\varphi$.
\end{proof}

Having identified $(\varphiY)^{-1}(\Kk(Y))$ with $I:=J_X \oplus \ker\varphi \oplus
K_\varphi^\infty$, since $\varphiY$ is injective, we can apply the graded version of Pimsner's
exact sequence \cite[Theorem~7.0.3]{PattersonHons} to the $(A \oplus K_\varphi^\infty)$--$(A
\oplus K_\varphi^\infty)$-correspondence $Y$. To avoid overly heavy notation in the resulting
sequences \eqref{eq:covariant}~and~\eqref{eq:contravariant}, we employ the following slight abuses
of notation: In the following diagrams, although $\iota:J_X\oplus\ker\varphi\to A$ is the
inclusion map, we write $[\iota]$ for the class $[I, \iota\oplus \id_{K_\varphi^\infty}, 0,
\alpha_I]\in KK(I, A \oplus K_\varphi^\infty)$; moreover, in these diagrams we write $[X]$ and
$[T]$ for the classes $[X, \varphi|_{J_X + \ker \varphi}, 0, \alpha_X]$ and $[T, \varphi_T|_I, 0,
\alpha_T]$ respectively. Both are elements of $KK(I, A \oplus K_\varphi^\infty)$ by letting
$K_\varphi^\infty$ act trivially on $X$ (on the left and right) and letting $A$ act trivially on
$T$ on the right (but not on the left). While this is slightly at odds with our usual use of, for
example, the notation $[X]$ corresponding to a Hilbert module $X$, the ambient $KK$-groups in the
diagrams should provide enough context to avoid confusion.
\begin{equation}\label{eq:covariant}
\parbox[c]{0.92\textwidth}{
\begin{tikzpicture}[xscale=0.8, yscale=0.6, >=stealth]
	\node (00) at (0,0) {$KK_1(B, \Oo_Y)$};
	\node (40) at (4.5,0) {$KK_1(B, A\oplus K_\varphi^\infty)$};
	\node (80) at (12,0) {$KK_1(B, J_X\oplus\ker\varphi\oplus K_\varphi^\infty)$};
	\node (82) at (12,2) {$KK_0(B, \Oo_X)$};
	\node (42) at (7.5,2) {$KK_0(B, A\oplus K_\varphi^\infty)$};
	\node (02) at (0,2) {$KK_0(B, J_X\oplus\ker\varphi\oplus K_\varphi^\infty)$};
	\draw[->] (02)-- node[above] {${\scriptstyle{\hatimes ([\iota] - [X] - [T])}}$} (42);
	\draw[->] (42)-- node[above] {${\scriptstyle i_*}$} (82);
	\draw[->] (82)--(80);
	\draw[->] (80)-- node[above] {${\scriptstyle{\hatimes ([\iota] - [X] - [T])}}$} (40);
	\draw[->] (40)-- node[above] {${\scriptstyle i_*}$} (00);
	\draw[->] (00)--(02);
\end{tikzpicture}}
\end{equation}
\begin{equation}\label{eq:contravariant}
\parbox[c]{0.92\textwidth}{
\begin{tikzpicture}[xscale=0.8, yscale=0.6, >=stealth]
	\node (00) at (0,0) {$KK_1(\Oo_Y, B)$};
	\node (40) at (4.5,0) {$KK_1(A\oplus K_\varphi^\infty, B)$};
	\node (80) at (12,0) {$KK_1(J_X\oplus\ker\varphi\oplus K_\varphi^\infty, B)$};
	\node (82) at (12,2) {$KK_0(\Oo_X, B)$};
	\node (42) at (7.5,2) {$KK_0(A\oplus K_\varphi^\infty, B)$};
	\node (02) at (0,2) {$KK_0(J_X\oplus\ker\varphi\oplus K_\varphi^\infty, B)$};
	\draw[<-] (02)-- node[above] {${\scriptstyle{([\iota] - [X] - [T])} \hatimes}$} (42);
	\draw[<-] (42)-- node[above] {${\scriptstyle i^*}$} (82);
	\draw[<-] (82)--(80);
	\draw[<-] (80)-- node[above] {${\scriptstyle{([\iota] - [X] - [T]) \hatimes }}$} (40);
	\draw[<-] (40)-- node[above] {${\scriptstyle i^*}$} (00);
	\draw[<-] (00)--(02);
\end{tikzpicture}}
\end{equation}

\subsection{The covariant exact sequence}

In this section we will use~\eqref{eq:covariant} to recover our exact sequence describing the
graded $K$-theory $KK(\CC, \Oo_X)$.


\begin{prp}\label{Prop:T tensor is shift}
For any graded $C^*$-algebra $B$, define the
\begin{align*}\textstyle
\textstyle U:KK(B,J_X)\oplus KK(B,\ker\varphi)&
\textstyle \oplus\big(\bigoplus_{n=1}^\infty KK(B,\ker\varphi)\big)\\
&\textstyle\to KK(B,A)\oplus\big(\bigoplus_{n=1}^\infty KK(B,\ker\varphi)\big),
\end{align*}
by $U(j,a,(f_1,f_2,\dots))=(0,(a,f_1,f_2,\dots))$. Let $\omega : \bigoplus_{n=1}^\infty
KK(B,\ker\varphi) \to KK(B,K_\varphi^\infty)$ be the canonical homomorphism of
\eqref{omegamap}, and let $[T]=[T,\varphiY,0,\alpha_T]$ be the Kasparov class of
$T$. Then the following diagram commutes
\begin{equation*}
\xymatrix{\hspace{-1cm}
	KK(B,J_X\oplus\ker\varphi\oplus K_\varphi^\infty)\ar[r]^{\ \ \ \  \otimes [T]\ \ \ \ } & KK(B,A\oplus K_\varphi^\infty)\\
    \displaystyle KK(B,J_X)\oplus KK(B,\ker\varphi)\oplus\bigoplus_{n=1}^{\infty}KK(B,\ker\varphi) \ar[r]^{\ \ \ \ \ \ \ \ \ \ \ \ \ \  U\ \ \ \ } \ar[u]_{\omega} &\displaystyle KK(B,A)\oplus\bigoplus_{n=1}^{\infty}KK(B,\ker\varphi) \ar[u]^{\omega}\\
}
\end{equation*}
\end{prp}
\begin{proof}
Denote the $n$th copy of $\ker\varphi$ in $K_\varphi^\infty$ by $K_n$, and let $K_0$ be the copy of
$\ker\varphi$ in $A$. Since classes of Kasaprov $B$--$J_X$-modules and Kasparov $B$--$K_n$-modules
generate $KK(B,J_X\oplus K_0\oplus K_\varphi^\infty)$ it suffices to show that the diagram commutes on such
elements.

Let $[J,\psi_J, F_j, \alpha_J]$ be a $B$--$J_X$-module. Since the Fredholm operator defining the
Kasparov class $[T]$ is zero, there is a 0-connection $F_J\hatimes 1$ for $J$ such that
\[
    [J,\psi_J, F_j, \alpha_J]\hatimes[T,\varphiY,0,\alpha_T]=[J\hatimes T, \psi_J\hatimes 1, F_J\hatimes 1,\alpha_J\hatimes\alpha_T].
\]
We claim this is the zero module. To see this, fix $j\in J$ and $f\in T$, and use
\cite[Proposition~2.31]{tfb} to write $j=k\cdot\langle k,k\rangle$. Using that $\langle
k,k\rangle\in\ker\varphi^\perp$ at the last equality, we compute
\[
j\hatimes f
    =k\cdot\langle k,k\rangle\hatimes f
    =k\hatimes\varphiY(\langle k,k\rangle)f
    =k\hatimes(0,\langle k,k\rangle f_1,0)
    =0.
\]
Since simple tensors span $J\hatimes T$, it follows the tensor product is zero. Hence
$\omega([J])\hatimes[T]=\omega \circ U([J])=0$.

Now consider a Kasparov $B$--$K_n$-module $(Z_n,\psi_n,F_n,\alpha_n)$. Since $Z_n$ is a
$B$--$K_n$-correspon\-dence there exists $a_z\in \ker \varphi$ such that $\langle z,z\rangle_i=a_z\delta_{i,n}$ for all $z\in Z_n$. Hence for $z\in Z_n$ and $f\in T$ we have
\[
\langle z\hatimes f,z\hatimes f\rangle_i
    =\langle f,\varphiY(\langle z,z\rangle)f\rangle_i
    =f^*_i\langle z,z\rangle_{i-1}f_i
    =f^*_ia_z\delta_{i-1,n}f_i,
\]
which is non-zero only if $i=n+1$. Hence $\langle z\hatimes f,z\hatimes f\rangle \in K_{n+1}$.
Thus $\langle y,y\rangle\in K_{n+1}$ for all $y\in Z_n\hatimes T$. With $j_n\colon KK(B,K_n)\to KK(B,J_X\oplus K_0\oplus K_\varphi^\infty)$ denoting the canonical inclusion,
\[
   \omega \circ j_n([Z_n])\hatimes[T]= \omega\big((\overbrace{0,\dots, 0}^{n\text{ terms}},[Z_n\hatimes T],0,\dots)\big)=\omega\circ j_{n+1}([Z_n\hatimes T])
\]
There is an isomorphism $Z_n\hatimes T \cong Z_n$ that carries an elementary tensor $z\hatimes f$
to $z\cdot f$. Thus, for $f_n= j_n [Z_n]$,
\begin{equation*}
\omega(U(f_n))=\omega(j_{n+1}[Z_n])= \omega \circ j_n([Z_n])\hatimes[T]=\omega(f_n)\hatimes[T]. \qedhere
\end{equation*}
\end{proof}

\begin{thm}\label{Theorem: First half of main thm A}
Let $(A,\alpha_A)$ be a graded separable nuclear $C^*$-algebra. Let $X$ be an essential graded
$A$--$A$-correspondence with left action $\varphi$. Let
$J_X=\varphi^{-1}(\Kk(X))\cap\ker\varphi^\perp$, and let $\iota_{J_X} : J_X\to A$ be the inclusion
map. Then there is an exact sequence
\[	
\begin{tikzpicture}[yscale=0.6, >=stealth]
	\node (00) at (0,0) {$KK_1(\CC, \Oo_X)$};
	\node (40) at (6,0) {$KK_1(\CC, A)$};
	\node (80) at (12,0) {$KK_1(\CC, J_X)$};
	\node (82) at (12,2) {$KK_0(\CC, \Oo_X)$};
	\node (42) at (6,2) {$KK_0(\CC, A)$};
	\node (02) at (0,2) {$KK_0(\CC, J_X)$};
	\draw[->] (02)-- node[above] {${\scriptstyle{\hatimes_A ([\iota_{J_X}] - [X])}}$} (42);
	\draw[->] (42)-- node[above] {${\scriptstyle i_*}$} (82);
	\draw[->] (82)--(80);
	\draw[->] (80)-- node[above] {${\scriptstyle{\hatimes_A ([\iota_{J_X}] - [X])}}$} (40);
	\draw[->] (40)-- node[above] {${\scriptstyle i_*}$} (00);
	\draw[->] (00)--(02);
	\end{tikzpicture}
\]
\end{thm}
\begin{proof}
By \cite[Lemma~4.2]{MT} the left action $\varphiY$ on $Y=X\oplus T$ is injective. Let
$I:=(\varphiY)^{-1}(\Kk(Y))$ and let $\iota_I : I\hookrightarrow A\oplus K_\varphi^\infty$ be the
inclusion map. Consider the resulting exact sequence~\eqref{eq:covariant}. Let
$P:KK(\CC,(J_X+\ker\varphi)\oplus K_\varphi^\infty)\to KK(\CC,J_X)$ be the projection map given by
$P[Z,\psi,F,\alpha_Z]=[Z\cdot J_X,\psi,F,\alpha_Z]$, and let $\ell:KK(\CC,A)\to KK(\CC,A\oplus
K_\varphi^\infty)$ be the inclusion map. Consider the following ten-term diagram
with~\eqref{eq:covariant} as its central 6-term rectangle.
\[
	\begin{tikzpicture}[xscale=0.9, yscale=0.4, >=stealth]
	\node (00) at (0,0) {$KK_1(\CC, \Oo_Y)$};
	\node (40) at (4.5,0) {$KK_1(\CC, A\oplus K_\varphi^\infty)$};
	\node (80) at (12,0) {$KK_1(\CC, J_X\oplus\ker\varphi\oplus K_\varphi^\infty)$};
	\node (82) at (12,4) {$KK_0(\CC, \Oo_Y)$};
	\node (42) at (7.5,4) {$KK_0(\CC, A\oplus K_\varphi^\infty)$};
	\node (02) at (0,4) {$KK_0(\CC, J_X\oplus\ker\varphi\oplus K_\varphi^\infty)$};
	
	\node (04) at (0,8) {$KK_0(\CC,J_X)$};
	\node (44) at (7.5,8) {$KK_0(\CC,A)$};
	\node (8-4) at (12,-4) {$KK_1(\CC,J_X)$};
	\node (4-4) at (4.5,-4) {$KK_1(\CC,A)$};
	
	\draw[->] (02)-- node[above] {${\scriptstyle{\hatimes_A ([\iota_I] - [X] - [T])}}$} (42);
	\draw[->] (42)-- node[above] {${\scriptstyle i_*}$} (82);
	\draw[->] (82)-- node[right] {${\scriptstyle\partial}$} (80);
	\draw[->] (80)-- node[above] {${\scriptstyle{\hatimes_A ([\iota_I] - [X] - [T])}}$} (40);
	\draw[->] (40)-- node[above] {${\scriptstyle i_*}$} (00);
	\draw[->] (00)-- node[left] {${\scriptstyle\partial}$} (02);
	
	\draw[->] (02)-- node[left] {{$\scriptstyle P$}} (04);
	\draw[->] (04)-- node[above] {${\scriptstyle\hatimes_A([\iota_{J_X}]-[X])}$} (44);
	\draw[->] (44)-- node[right] {{$\scriptstyle\ell$}} (42);
	\draw[->] (80)-- node[right] {{$\scriptstyle P$}} (8-4);
	\draw[->] (8-4)-- node[above] {${\scriptstyle\hatimes_A([\iota_{J_X}]-[X])}$} (4-4);
	\draw[->] (4-4)-- node[left] {{$\scriptstyle\ell$}} (40);
	\end{tikzpicture}
\]
We show that the sequence
\begin{equation}\label{Eq: newSequence}
\parbox[c]{0.9\textwidth}{\mbox{}\hfill
	\begin{tikzpicture}[xscale=0.9,yscale=0.8, >=stealth]
	\node (00) at (0,0) {$KK_1(\CC, \Oo_Y)$};
	\node (40) at (6,0) {$KK_1(\CC, A)$};
	\node (80) at (12,0) {$KK_1(\CC, J_X)$};
	\node (82) at (12,2) {$KK_0(\CC, \Oo_Y)$};
	\node (42) at (6,2) {$KK_0(\CC, A)$};
	\node (02) at (0,2) {$KK_0(\CC, J_X)$};
	\draw[->] (02)-- node[above] {${\scriptstyle{\hatimes_A ([\iota_{J_X}] - [X])}}$} (42);
	\draw[->] (42)-- node[above] {${\scriptstyle i_* \circ \ell}$} (82);
	\draw[->] (82)-- node[right] {${\scriptstyle P\circ\partial}$}(80);
	\draw[->] (80)-- node[above] {${\scriptstyle{\hatimes_A ([\iota_{J_X}] - [X])}}$} (40);
	\draw[->] (40)-- node[above] {${\scriptstyle i_* \circ \ell}$} (00);
	\draw[->] (00)-- node[left] {${\scriptstyle P\circ\partial}$} (02);
	\end{tikzpicture}\hfill\mbox{}}
\end{equation}
obtained by traversing the outside of the ten-term diagram is exact. First we show that~\eqref{Eq:
newSequence} is exact at $KK_*(\CC,\Oo_Y)$. We claim that $\ker(P \circ
\partial) = \ker\partial$. To see this, observe that $\Img(\partial) = \ker(\cdot \hatimes
([\iota_I] - [X] - [T]))$ by exactness of~\eqref{eq:covariant}. Identifying direct sums as in
Proposition~\ref{prp:KK(C,A) countable additivity}, suppose that
$(j,a,f)\hatimes([\iota_I]-[X]-[T])=0$ for some $j\in KK(\CC,J_X), a\in KK(\CC,\ker\varphi)$, and
$f=(f_1,f_2,\dots)\in \bigoplus_{n=1}^\infty KK(\CC,\ker\varphi)$. Since $\iota_I$ is just the
inclusion of $I$ into $A \oplus K_\varphi^\infty$, the map $(\iota_I)_* := \cdot \hatimes
[\iota_I] : KK(\CC, I) \to KK(\CC, A)$ is just the natural inclusion. Hence under the
identification of Proposition~\ref{ideal.I}, $j\hatimes[\iota_I]$, $a\hatimes[\iota_I]$, and
$f\hatimes[\iota_I]$ are the natural images of $j$, $a$ and $f$ in $KK(\CC,A\oplus
K_\varphi^\infty)$. Hence $(j, a, f) \hatimes [\iota_I] = (j+a, f)$.

Since $\ker\varphi$ acts trivially on $X$, we have $a\hatimes [X]=f\hatimes[X]=0$. Hence
$(j,a,f)\hatimes[X]=(j\hatimes[X],0,0)$. Using Proposition~\ref{Prop:T tensor is shift}, $(j,a,(f_1,f_2,\dots))\hatimes [T]=(0,(a,f_1,f_2,\dots))$, so
\begin{equation}\label{Eq: finiteSeq}
0=(j,a,f)\hatimes([\iota_I]-[X]-[T])=(j+a-j\hatimes[X],f_1-a,f_2-f_1,\dots,f_{n}-f_{n-1},\dots).
\end{equation}
Proposition~\ref{prp:KK(C,A) countable additivity} shows that $f=(f_1,f_2,\dots) \in
\bigoplus_{n=1}^\infty KK(\CC,\ker\varphi)$, so there exists $N\in N$ such that $f_n=0$ for all
$n> N$. Hence~\eqref{Eq: finiteSeq} becomes
\[
0=(j+a-j\hatimes[X],f_1-a,f_2-f_1,\dots,f_{N}-f_{N-1},-f_N,0,0,\dots)
\]
forcing $f_N=0$. Continuing recursively down the sequence we have $f_k=0$ for each $k\leq N$, and
$a=0$. Hence $(j,a,f)=(j,0,0)$. We conclude
\begin{equation}
\label{kerpartial}
\Img(\partial) \subseteq KK(\CC, J_X) \oplus 0 \oplus 0,
\end{equation}
so
$P|_{\ker\partial}$ is injective, and therefore $\ker(P\circ\partial) = \ker(\partial)$.

To establish exactness of~\eqref{Eq: newSequence} at $KK_*(\CC, \Oo_Y)$, it now suffices to show
that $\Img(\iota_X \circ \ell) = \Img(\iota_X)$. For this, we first claim that
\begin{equation}\label{Eq: img of tensor iota-Y}
	\Img(\hatimes([\iota_I]-[X]-[T])\big|_{KK(\CC,\ker\varphi\oplus T)})=\Big\{\Big(-\sum f_j,f\Big):f\in KK(\CC,T)\Big\}.
\end{equation}
For all $a\in KK(\CC,\ker\varphi)$ and $f\in KK(\CC,T)$, the product
$(0,a,f)\hatimes([\iota_I]-[X]-[T])$ is of the form $\big(-\sum g_j,g\big)$, where
$g_j=f_j-f_{j-1}$ for $j>1$, and $g_1=f_1-a$. Conversely, we have $\big(-\sum
f_j,f\big)=(0,a,g)\hatimes([\iota_I]-[X]-[T])$ for $g_j=-\sum_{k=0}^{n-j-1}f_{n-k}$ and
$a=-\sum_{k=1}^n f_k$, where $n$ is the index of the last non-zero entry of $f$. This proves
\eqref{Eq: img of tensor iota-Y}. By exactness of the inner rectangle~\eqref{eq:covariant} of the
ten-term diagram, we have $\Img(\hatimes([\iota_I]-[X]-[T])\big|_{KK(\CC,\ker\varphi\oplus
T)})\subseteq\Img(\hatimes([\iota_I]-[X]-[T]))=\ker i_*$, so we deduce that $i_*\big(-\sum
f_j,f\big)=0$ for all $f\in KK(\CC,T)$. Now suppose $i_*(a,f)\in\Img(i_*)$. Then
\[
i_*(a,f)=i_*(a+\sum f_j,0)+i_*(-\sum f_j,f)=i_*(a+\sum f_j,0)=i_*\circ\ell(a+\sum f_j).
\]
Hence $\Img(i_*\circ\ell)=\Img(i_*)=\ker(\partial)$ by exactness of~\eqref{eq:covariant}. This
completes the proof of exactness of~\eqref{Eq: newSequence} at $KK_*(\CC, \Oo_Y)$.

We now establish exactness of~\eqref{Eq: newSequence} at $KK_*(\CC,J_X)$. We have already
demonstrated in \eqref{kerpartial} that
$\ker(\hatimes([\iota_I]-[X]-[T]))=\ker(\hatimes([\iota_I]-[X]-[T])|_{KK(\CC,J_X)})$, and further by Proposition~\ref{Prop:T tensor is shift}
we have that $\cdot \hatimes([\iota_I]-[X]-[T])|_{KK(B,J_X)} = \cdot\hatimes([\iota_{J_X}]-[X])$.
Hence
\[
\Img(i_*\circ\ell)=\Img(i_*)=\ker(\hatimes([\iota_I]-[X]-[T]))=\ker(\hatimes([\iota_{J_X}]-[X])),
\]
giving exactness at $KK_*(\CC,J_X)$.

Next, we establish exactness of~\eqref{Eq: newSequence} at $KK_*(\CC,A)$. By definition we have
\[
\ker(i_*\circ\ell)=\{a\in KK(\CC,A),i_*(a,0)=0\}=\ker(i_*)\cap( KK(\CC,A)\oplus \{0\}).
\]
Suppose that for $j\in KK(\CC,J_X), a\in KK(\CC,\ker\varphi)$ and $f\in\bigoplus_{n=1}^\infty KK(\CC,\ker\varphi)$ we
have $(j,a,f)\hatimes([\iota_I]-[X]-[T])=(b,0)$ for some $b\in KK(\CC,A)$. Then
\begin{align*}
(j,a,f)\hatimes([\iota_I]-[X]-[T])
    &=(j+a-j\hatimes[X],f_1-a,f_2-f_1,\dots,-f_N,0,\dots)\\
    &=(b,0,\dots),
\end{align*}
where $N\in\NN$ is the index of the last non-zero component of $f$. We have $f_N=0$, and so
recursively, $f_j=0$ for each $j>0$, and $a=0$. Hence
$(j,a,f)\hatimes([\iota_I]-[X]-[T])=(j,0)\hatimes([\iota_{J_X}]-[X])\in\Img(\hatimes([\iota_{J_X}]-[X]))$.
Thus $\Img(\hatimes([\iota_I]-[X]-[T]))\cap(KK(\CC,A)\oplus \{0\})=\Img([\iota_{J_X}]-[X])$, and
so by exactness of~\eqref{eq:covariant}, we deduce that
$\Img(\hatimes([\iota_{J_X}]-[X]))=\ker(i_*\circ\ell)$. This proves exactness of~\eqref{Eq:
newSequence}.

The inclusion $i : A \oplus K_\varphi^\infty \to \Oo_Y$ is nondegenerate and so extends to a
homomorphism $\tilde{i} : \Mm(A \oplus K_\varphi^\infty) \to \Mm(\Oo_Y)$. Theorem~4.3 of \cite{MT}
shows that $Q := \tilde{i}(1_{\Mm(A)}) \in \Mm(\Oo_Y)$ is a full projection and that $Q \Oo_Y Q
\cong \Oo_X$. Since $Q$ is trivially graded with respect to the grading on $A \oplus
K_\varphi^\infty$, the space $\Oo_Y Q$ is a graded imprimitivity $\Oo_Y$--$\Oo_X$-module. So
$(\Oo_Y, \alpha_{\Oo_Y})$ and $(\Oo_X, \alpha_{\Oo_X})$ are $KK$-equivalent as discussed in
Section~\ref{sec:KKBasics}. In particular, $KK_*(\CC,\Oo_X)\cong KK_*(\CC,\Oo_Y)$. Now $ i_* \circ
\ell\colon KK(\CC,A)\to KK(\CC,\Oo_Y)$ restricts to $ i_* \colon KK(\CC,A)\to KK(\CC,\Oo_X)$,
giving the desired sequence.
\end{proof}

\subsection{The contravariant exact sequence}

We now use similar techniques to those used in the preceding subsection to obtain an exact
sequence describing $KK(\Oo_X, B)$ using the contravariant exact
sequence~\eqref{eq:contravariant}.

\begin{prp}\label{Prop:T tensor description} Define
\begin{align*}
\overline{U}: KK(A, B)\oplus{}&\left(\prod_{n=1}^\infty KK(\ker\varphi, B)\right)\\
     &\to KK(J_X, B)\oplus KK(\ker\varphi, B)\oplus\left(\prod_{n=1}^\infty KK(\ker\varphi, B)\right)
\end{align*}
by $\overline{U}(a,f_1,f_2,\dots)=(0,f_1,f_2,\dots)$. Let $\overline{\zeta} : KK(K_\varphi^\infty,
B)\to \prod_{n=1}^\infty KK(\ker\varphi, B)$ be the isomorphism discussed at the beginning of
Section~\ref{sec:additive}, and let $[T]=[T,\varphiY,0,\alpha_T]$ be the class of the
module associated to $T$. Then the following diagram commutes.
\begin{equation*}
\xymatrix{\hspace{-1cm}
	KK((J_X\oplus\ker\varphi)\oplus K^\infty_\varphi, B)\ar[d]_{\zeta}& \ar[l]^{\ \ \ \  [T] \hatimes\ \ \ \ }KK(A\oplus K^\infty_\varphi, B)\ar[d]^{\overline{\zeta}}\\
	\displaystyle KK(J_X, B)\oplus KK(\ker\varphi, B)\oplus\prod_{n=1}^{\infty}KK(B,\ker\varphi) & \ar[l]^{\ \ \ \ \ \ \ \ \ \ \ \ \ \  \overline{U}\ \ \ \ } \displaystyle KK(A,B)\oplus\prod_{n=1}^{\infty}KK(\ker\varphi, B)\\
}
\end{equation*}
\end{prp}
\begin{proof}
As discussed at the beginning of Section~\ref{sec:additive}, if $\iota_i : \ker \varphi \to
K^\infty_\varphi$ is the inclusion into the $i$th coordinate, then the right-hand map $\zeta$
carries the class $[Z]$ of a Kasparov $(A \oplus K^\infty_\varphi)$--$B$-module to $\big([A \cdot
Z], ([\iota_i(\ker \varphi)\cdot Z])^\infty_{i=1}\big)$. Likewise, the right-hand map takes $[W]$
to $\big([J_X \cdot W], [\ker \varphi \cdot W], ([\iota_i(\ker \varphi) \cdot
W])^\infty_{i=1}\big)$.

Regard $T$ as a right $(J_X \oplus K^\infty_\varphi)$-module, and take $W = T \hatimes Z$. The
left action of $K^\infty_\varphi$ on $Z$ is given by the inclusion $K^\infty_\varphi
\hookrightarrow A \oplus K^\infty_\varphi$, and so $T \hatimes (A \oplus 0)\cdot Z = 0$. For each
$i \ge 1$, we have $T \hatimes (\iota_i(\ker \varphi)) \cdot Z \cong \ker \varphi
\otimes_{\ker\varphi} (\iota_i(\ker \varphi)) \cdot Z \cong \iota_i(\ker \varphi) \cdot Z$ as a
right module. The left action of $J_X \oplus \ker\varphi$ on $T \hatimes (0 \oplus
\iota_1(\ker\varphi)) \cdot Z$ restricts to the zero action of $J_X$ because $J_X \subseteq
\ker\varphi^\perp$, and restricts to the standard action of $\ker\varphi$. So we see that $W$ is
isomorphic to $0 \oplus \bigoplus^\infty_{i=1} \iota_i(\ker\varphi) \cdot Z$ as a right-Hilbert
$J_X \oplus \ker\varphi \oplus K^\infty_\varphi$-module. This isomorphism preserves gradings
because $\iota$ is a graded homomorphism.

In particular, $\iota_{J_X} \cdot W = 0$, each
$\iota_{K^\infty_\varphi}(\iota_i(\ker\varphi))\cdot W \cong \iota_{i+1}(\ker\varphi)\cdot W$, and
 $\iota_{\ker\varphi} \cdot W \cong \iota_1(\ker\varphi)\cdot Z$. Thus $\zeta([W]) = (0,
[\iota_1(\ker\varphi)\cdot Z], [\iota_2(\ker\varphi)\cdot Z], \cdots)$. Since $\zeta([Z]) = ([A
\cdot Z], [\iota_1(\ker\varphi)\cdot Z], [\iota_2(\ker\varphi)\cdot Z], \cdots)$, the result
follows.
\end{proof}

\begin{thm}\label{Theorem: Second half of main thm A}
Let $(A,\alpha_A),(B,\alpha_B)$ be a graded separable $C^*$-algebras, and suppose $A$ is nuclear.
Let $X$ be an essential graded $A$--$A$-correspondence with left action $\varphi$. Let
$J_X=\varphi^{-1}(\Kk(X))\cap\ker\varphi^\perp$, and let $\iota_{J_X}:J_X\to A$ be the inclusion map.
Then there is an exact sequence
\[	
\begin{tikzpicture}[yscale=0.6, >=stealth]
	\node (00) at (0,0) {$KK_1(\Oo_X, B)$};
	\node (40) at (6,0) {$KK_1(A, B)$};
	\node (80) at (12,0) {$KK_1(J_X, B)$};
	\node (82) at (12,2) {$KK_0(\Oo_X, B)$};
	\node (42) at (6,2) {$KK_0(A, B)$};
	\node (02) at (0,2) {$KK_0(J_X, B)$};
	\draw[<-] (02)-- node[above] {${\scriptstyle{([\iota_{J_X}] - [X]) \hatimes_A}}$} (42);
	\draw[<-] (42)-- node[above] {${\scriptstyle i^*}$} (82);
	\draw[<-] (82)--(80);
	\draw[<-] (80)-- node[above] {${\scriptstyle{([\iota_{J_X}] - [X]) \hatimes_A}}$} (40);
	\draw[<-] (40)-- node[above] {${\scriptstyle i^*}$} (00);
	\draw[<-] (00)--(02);
	\end{tikzpicture}
\]
\end{thm}
\begin{proof}
The argument is similar to that of Theorem~\ref{Theorem: First half of main thm A}. Since the left
action $\varphi^{A \oplus K^\infty_\varphi}$ on $Y$ is injective, if we write $I := (\varphi^{A
\oplus K^\infty_\varphi})^{-1}(\Kk(Y))$ and $\iota_I : I \to A \oplus K^\infty_{\varphi}$ for the
inclusion, then the exact sequence~\eqref{eq:contravariant} fits as the central rectangle in the
following diagram whose top and bottom rectangles commute:
\[
\begin{tikzpicture}[xscale=0.9, yscale=0.4, >=stealth]
	\node (00) at (0,0) {$KK_1(\Oo_Y, B)$};
	\node (40) at (4.5,0) {$KK_1(A\oplus K^\infty_\varphi, B)$};
	\node (80) at (12,0) {$KK_1(J_X\oplus\ker\varphi\oplus K^\infty_\varphi, B)$};
	\node (82) at (12,4) {$KK_0(\Oo_Y, B)$};
	\node (42) at (7.5,4) {$KK_0(A\oplus K^\infty_\varphi, B)$};
	\node (02) at (0,4) {$KK_0(J_X\oplus\ker\varphi\oplus K^\infty_\varphi, B)$};
	
	\node (04) at (0,8) {$KK_0(J_X, B)$};
	\node (44) at (7.5,8) {$KK_0(A, B)$};
	\node (8-4) at (12,-4) {$KK_1(J_X, B)$};
	\node (4-4) at (4.5,-4) {$KK_1(A, B)$};
	
	\draw[<-] (02)-- node[above] {${\scriptstyle{([\iota_I] - [X] - [T]) \hatimes_A}}$} (42);
	\draw[<-] (42)-- node[above] {${\scriptstyle i^*}$} (82);
	\draw[<-] (82)-- node[right] {${\scriptstyle\partial}$} (80);
	\draw[<-] (80)-- node[above] {${\scriptstyle{([\iota_I] - [X] - [T]) \hatimes_A}}$} (40);
	\draw[<-] (40)-- node[above] {${\scriptstyle i^*}$} (00);
	\draw[<-] (00)-- node[left] {${\scriptstyle\partial}$} (02);
	
	\draw[<-] (02)-- node[left] {{$\scriptstyle \ell$}} (04);
	\draw[<-] (04)-- node[above] {${\scriptstyle([\iota_{J_X}]-[X]) \hatimes_A}$} (44);
	\draw[<-] (44)-- node[right] {{$\scriptstyle P$}} (42);
	\draw[<-] (80)-- node[right] {{$\scriptstyle \ell$}} (8-4);
	\draw[<-] (8-4)-- node[above] {${\scriptstyle([\iota_{J_X}]-[X]) \hatimes_A}$} (4-4);
	\draw[<-] (4-4)-- node[left] {{$\scriptstyle P$}} (40);
\end{tikzpicture}
\]
To prove the result, we show that the six-term sequence consisting of the six extreme points of
this diagram is exact; the result will again follow from the graded Morita equivalence of $\Oo_X$
and $\Oo_Y$.

Throughout this proof, without further comment, we identify $KK_*(J_X \oplus \ker\varphi \oplus
K^\infty_\varphi, B)$ with $KK_*(J_X, B) \oplus KK_*(\ker\varphi, B) \oplus \prod^\infty_{i=1}
KK_*(\ker\varphi, B)$ and we identify $KK_*(A \oplus K^\infty_\varphi, B)$ with $KK_*(A, B) \oplus
\prod^\infty_{i=1} KK_*(\ker\varphi, B)$ as discussed at the beginning of
Section~\ref{sec:additive}.

For exactness at $KK_*(A, B)$, observe that $\Img(P \circ i^*) = P\big(\ker(([\iota_I] - [X] -
[T]) \hatimes \cdot)\big)$. By Proposition~\ref{Prop:T tensor description}, and using that $\ker
\varphi$ annihilates $X$, we see that for any $(a, j_1, j_2, \cdots) \in KK_*(A \oplus
K^\infty_\varphi, B)$, we have
\[
([\iota_I] - [X] - [T]) \hatimes_A (a, j_1, j_2, \dots)
    = (([\iota_{J_X}] - [X]) \hatimes a, [\ker\varphi] \hatimes a - j_1, j_1 - j_2, \dots).
\]
So $\ker(([\iota_I] - [X] - [T]) \hatimes \cdot)$ is the set of sequences $(a, j, j, j, \dots)$
such that $([\iota_{J_X}] - [X]) \hatimes_A a = 0$ and $j = [\ker\varphi] \hatimes a$. In
particular, $P\big(\ker(([\iota_I] - [X] - [T]) \hatimes \cdot)\big) = \ker([\iota_{J_X}] - [X])$
as required.

For exactness at $KK_*(J_X, B)$ first observe that $\ell$ is given by $\ell([j]) = ([j], 0, 0,
\dots)$. So $\ker(\partial \circ \ell) = \{[j] : ([j], 0, 0, \dots) \in \ker(\partial)\} = \{([j],
0, 0, \dots) : [j] \in \Img(([\iota_I] - [X] - [T]) \hatimes_A \cdot)\}$. By the description of
the map $([\iota_I] - [X] - [T]) \hatimes \cdot$ in the preceding paragraph, we see that if $([j],
0, 0, \cdots) = ([\iota_I] - [X] - [T]) \hatimes_A (a, j_1, j_2, \cdots)$, then $[j] =
([\iota_{J_X}] - [X]) \hatimes_A a$. Conversely, given $a \in A$, since the rectangles involving
$P$ and $\ell$ commute and since the maps $P$ are surjective, $\ell(([\iota_{J_X}] - [X])
\hatimes_A a) \in \Img(([\iota_I] - [X] - [T]) \hatimes_A \cdot) = \ker(\partial)$, and so the
image of $([\iota_{J_X}] - [X]) \hatimes_A$ is contained in the kernel of $\partial \circ \ell$.

It remains to establish exactness at $KK_*(\Oo_Y, B)$. By exactness of~\eqref{eq:contravariant},
$\Img(i^*) = \ker\big(([\iota_I] - [X] - [T]) \hatimes \cdot\big)$. As we saw earler, this is the
collection of sequences $(a, j, j, j, \dots)$ such that $([\iota_{J_X}] - [X]) \hatimes_A a = 0$
and $j = [\ker\varphi] \hatimes a$. In particular, if $(P \circ i^*)(x) = 0$, then $i^*(x) = (0,
j, j, \cdots)$ with $j = [\ker\varphi] \hatimes 0 = 0$, and hence $i^*(x) = 0$. That is, $\ker(P
\circ i^*) = \ker(i^*)$. Since~\eqref{eq:contravariant} is exact, it now suffices to show that
$\Img(\partial \circ \ell) = \Img \partial$. We clearly have $\Img(\partial \circ \ell) = \Img
\partial$, so it suffices to show the reverse containment. For this, fix $\theta = (j_X, j_0, j_1,
j_2, \dots) \in KK_*(J_X \oplus \ker\varphi \oplus K^\infty_\varphi, B)$, so that
$\partial(\theta)$ is a typical element of $\Img(\partial)$. Consider the element $\eta := (0,
-j_0, -j_0 - j_1, -j_0 - j_1 - j_2, \cdots) \in KK_*(A \oplus K^\infty_\varphi, B)$. We have
$([\iota_I] - [X] - [T]) \hatimes \eta = (0, j_0, j_1, j_2 \cdots)$. In particular, $\theta =
([\iota_I] - [X] - [T]) \hatimes \eta + (j_X, 0, 0, 0 \cdots) = ([\iota_I] - [X] - [T]) \hatimes
\eta + \ell(j_X)$. Since~\eqref{eq:contravariant} is exact, $\partial(\theta) =
\partial(([\iota_I] - [X] - [T]) \hatimes \eta) + \partial(\ell(j_X)) = \partial\circ\ell(j_X)$.
So $\partial(\theta) \in \Img(\partial\circ\ell)$ as required.
\end{proof}


\begin{thebibliography}{00}
\bibitem{BHRS} T. Bates, J. Hong, I. Raeburn, and W. Szyma\'nski, \emph{The ideal structure of the
    $C^*$-algebras of infinite graphs}, Illinois J. Math. \textbf{46} (2002), 1159--1176.

\bibitem{B} B. Blackadar, $K$-theory for operator algebras.  MSRI Publications vol. 5,
    \emph{Cambridge University Press}, 1998.

\bibitem{BDF1} L.~G. Brown, R.~G. Douglas and P.~A. Fillmore, \emph{Extensions of
  {$C^*$}-algebras, operators with compact self-commutators, and
  {$K$}-homology}, Bull. Amer. Math. Soc., \textbf{79} (1973), 973--978.

\bibitem{BDF2} L.~G. Brown, R.~G. Douglas, and P.~A. Fillmore, \emph{Unitary equivalence
  modulo the compact operators and extensions of {$C^*$}-algebras}, in
  Proceedings of a Conference on Operator Theory (Dalhousie Univ., Halifax,
  N.S., 1973), Berlin, 1973, Springer, 58--128. Lecture Notes in Math.,
  Vol. 345.

\bibitem{Crisp} T.~Crisp, \emph{Fredholm Modules over
	Graph {$C^*$}-Algebras}, Bull. Aust. Math. Soc. \textbf{92} (2015), 302--315

\bibitem{CunKri} J.~Cuntz and W.~Krieger, \emph{A class of {$C^*$}-algebras and topological {M}arkov
	chains}, Invent. Math.,  \textbf{56}(3) (1980), 251--268.

\bibitem{Cun81} J.~Cuntz, \emph{A class of {$C^*$}-algebras and topological {M}arkov chains.
  {II}. {R}educible chains and the {E}xt-functor for {$C^*$}-algebras}, Invent. Math.,  \textbf{63}(1) (1981), 25--40.

\bibitem{vD1} A. van Daele, \emph{Graded $K$-theory for Banach algebras I}, Quart.\ J.\ Math.\
    Oxford, \textbf{39} (1988), 185--199.

\bibitem{vD2} A. van Daele, \emph{Graded $K$-theory for Banach algebras II}, Pacific J.\ Math.,
    \textbf{134} (1988), 377--392.

\bibitem{DT} D.~Drinen and M.~Tomforde, \emph{Computing {$K$}-theory and {Ext} for graph {$C^*$}-algebras}, Illinois J. Math.,
    \textbf{46} (2002), 81--91.

\bibitem{Elliott} G.A. Elliott, \emph{On the {$K$}-theory of the {$C^*$}-algebra
	generated by a projective representation of a torsion-free discrete
	abelian group}, Monogr. Stud. Math., 17, Operator algebras and group representations,
{V}ol. {I} ({N}eptun, 1980), 157--184, Pitman, Boston, MA, 1984.

\bibitem{E} G. A. Elliott, \emph{On the classification of $C^*$-algebras of real rank
	zero}, J. reine angew. Math. \textbf{443} (1993), 179--219.

\bibitem{EL2} R.~Exel and M.~Laca, \emph{The $K$-theory of Cuntz-Krieger algebras
for infinite matrices}, $K$-Theory \textbf{19}(3) (2000), 251--268.

\bibitem{H1} U. Haag, \emph{On $\ZZ/2\ZZ$-graded KK-theory and its relation with the
	graded Ext-functor}, J. Operator Theory \textbf{42} (1999), 3--36.

\bibitem{H2} U. Haag, \emph{Some algebraic features of $\ZZ_2$-graded $KK$-theory}, K-Theory
    \textbf{13}  (1998), 81--108.

\bibitem{Karoubi} M. Karoubi, \emph{Alg\`ebres de {C}lifford et
	{$K$}-th\'eorie}, Ann. Sci. \'Ecole Norm. Sup. (4) \textbf{1} (1968), 161--270.

\bibitem{Kat} T. Katsura, \emph{On $C^*$-algebras associated with
    $C^*$-correspondences}, J. Funct. Anal. \textbf{217} (2004), 366--401.

\bibitem{K1} G.G. Kasparov, {\em The Operator $K$-Functor and Extensions of
	$C^*$-Algebras}, Math.\ USSR.\ Izv.\ {\bf 16} (1981),  513--572.

\bibitem{K2} G.G. Kasparov, {\em Equivariant $KK$-theory and the Novikov conjecture}, Invent.
    Math. \textbf{91} (1988), 147--201.

\bibitem{KPS6} A. Kumjian, D. Pask and A. Sims, \emph{Graded $C^*$-algebras, Graded $K$-theory,
    and Twisted $P$-graph $C^*$-algebras,} J. Operator Theory, \textbf{80} (2017), 295--348.

\bibitem{Lance} E.C. Lance, Hilbert {$C\sp *$}-modules, A toolkit for operator algebraists,
    Cambridge University Press, Cambridge, 1995, x+130.

\bibitem{MS} P.S. Muhly and B. Solel, \emph{Tensor algebras over $C^*$-correspondences:
    representations, dilations, and $C^*$-envelopes}, J. Funct. Anal. \textbf{158} (1998),
    389--457.

\bibitem{MT} P. Muhly and M. Tomforde, \emph{Adding tails to $C^*$-correspondences}, Doc. Math,
    \textbf{9} (2004), 79--106.

\bibitem{PR} D. Pask and I. Raeburn, \emph{ On the $K$-theory of Cuntz--Krieger
	algebras},
Publ. RIMS Kyoto Univ., {\bf 32} (1996), 415--443.

\bibitem{PattersonHons} Q. Patterson, \emph{Exact sequences in graded $KK$-theory for
    Cuntz-Pimsner algebras}, Honours thesis (2018), University of Wollongong (arXiv:2005.02187 [math.OA]).

\bibitem{P} M. Pimsner, \emph{A class of $C^*$-algebras generalising both Cuntz--Krieger
	algebras and crossed products by $\ZZ$}, Fields Inst. Comm. \textbf{12}, (1997)
    189--212.

\bibitem{Raeburn} I. Raeburn, \textit{Graph Algebras}, CBMS {\bf 103}, AMS, Providence, RI, 2005.

\bibitem{RSz} I. Raeburn  and W. Szyma{\'n}ski,
    \emph{Cuntz--{K}rieger algebras of infinite graphs and matrices}, Trans. Amer. Math.
    Soc. \textbf{356} (2004), 39--59.

\bibitem{tfb} I. Raeburn and D.P. Williams, Morita equivalence and continuous-trace {$C\sp
    *$}-algebras, American Mathematical Society, Providence, RI, 1998, xiv+327.

\bibitem{RRS} A. Rennie, D. Robertson and A. Sims, \emph{Groupoid Fell bundles for
	product systems over quasi-lattice ordered groups},
	Math.\ Proc.\ Cambridge Philos.\ Soc., {\bf 163} (2017), 561--580.

\bibitem{S} G. Skandalis, \emph{Exact sequences for the Kasparov groups of graded
	algebras}, Canad. J. Math. \textbf{37} (1985), 193--216.

\bibitem{Szym} W. Szyma{\'n}ski, \emph{Bimodules for Cuntz-Krieger algebras
of infinite matrices}, Bull. Austral. Math. Soc. \textbf{62}(1) (2000), 87--94.


\bibitem{Szym02} W. Szyma{\'n}ski, \emph{On semiprojectivity of {$C^*$}-algebras of directed
graphs}, Proc. Amer. Math. Soc. \textbf{130} (2002), 1391--1399.

\bibitem{TaylorHons} J. Taylor, \emph{(Up)Graded $K$-homology for graph algebras}, Honours thesis
    (2019), University of Wollongong.

\bibitem{Tom} M. Tomforde, \emph{Computing Ext for graph algebras}, J. Operator
Theory \textbf{49}(2) (2003), 363--387.

\bibitem{Yi} I.~Yi, \emph{{$K$}-theory and {$K$}-homology of {$C^*$}-algebras for row-finite
  graphs.}, Rocky Mountain J. Math. \textbf{37}(5) (2007), 1723--1742.
\end{thebibliography}
\end{document}